\documentclass[11pt]{article}
\usepackage{amssymb,latexsym,amsmath,amsbsy,amsthm,amsxtra,amsgen,graphicx,makeidx,dsfont,sidecap,comment,mathtools,caption}
\usepackage[utf8]{inputenc}
\usepackage{tikz}
\usetikzlibrary{decorations.pathreplacing}
\usepackage{upgreek}
\newfont{\msbm}{msbm10 at 11pt}

\newfont{\msbmsm}{msbm10 at 8pt}

\newtheorem{Theo}{Theorem}
\newtheorem{Lemma}[Theo]{Lemma}

\numberwithin{equation}{section}

\begin{document}
\title{The Yule-Kingman Nested Coalescent: Distribution of the Number of Lineages}
\author{Toni Gui}
\maketitle

\vspace{-.3in}
\begin{abstract}
We consider a model of a population in which individuals are sampled from different species. The Yule-Kingman nested coalescent describes the genealogy of the sample when each species merges with another randomly chosen species with a constant rate $c$ and each pair of lineages belonging to the same species coalesces independently at rate 1. We study the distribution of the number of individual lineages belonging to a given species and show that it converges weakly to a distribution $\mu_c^*$, which is the unique solution to a recursive distributional equation. Furthermore, the average number of individual lineages belonging to each species converges in probability to the mean of $\mu_c^*$.  \\


\end{abstract}


\section{Introduction}
Coalescent theory is the study of random processes in which particles, typically representing individual genes, can merge over time. Coalescent processes can help describe the genealogy of a population and are used to study how sampled genes originated from a common ancestor. Coalescent theory was developed in the early 1980s, and Kingman's work \cite{kingman_1982} is of primary importance. 

Kingman's coalescent \cite{KINGMAN1982235} is a famous stochastic process of coalescence. It traces back the ancestry of a gene sampled from a population which has a constant size through generations and undergoes neutral selection. 
Kingman's coalescent is a strong Markov process taking values in the set of the partitions of $\mathbb{N}^+=\{1,2,\ldots\}$. It initially starts from $\{\{1\},\{2\},\ldots\}$ and each pair of blocks merges independently at rate 1. A strong Markov process is called Kingman's $n$-coalescent if it starts from $\{\{1\},\ldots,\{n\}\}$ and each pair of blocks merges independently at rate 1.
Some results about Kingman's coalescent and its variants and their applications to theoretical population genetics are summarized in \cite{berestycki2009recent}. A striking feature of Kingman's coalescent is that it comes down from infinity. That is, starting from an infinite number of blocks, the number of blocks is almost surely finite after any positive amount of time in Kingman's coalescent. More formally, let $K_{\infty}(t)$ denote the number of blocks surviving to time $t$ in Kingman's coalescent. Then $P(K_{\infty}(t)<\infty\text{ for all } t>0)=1$.
Theorem 1 of \cite{Berestycki_2010} states that $\lim_{t\to0}tK_{\infty}(t)=2$ almost surely, which implies that the number of blocks in the Kingman's coalescent decays as $2/t$. Theorem 2 of \cite{Berestycki_2010} states that 
$$\lim_{s\to0}E\left[\sup_{t\in(0,s]}\left|\frac{tK_{\infty}(t)}{2}-1\right|\right]=0,$$
which implies that for any $\epsilon>0$, and for all $t$ sufficiently small,
$$E[K_{\infty}(t)]<\frac{2+2\epsilon}{t}.$$
Since $K_{\infty}(t)$ decreases as $t$ increases, we have $E[K_{\infty}(t)]$ is finite for all $t>0$.

A limitation of Kingman's coalescent is that it is only suitable for approximating the genealogy of a sample from one species, but cannot be used to explore the relationships among species. The problem of how to reconstruct species trees from gene trees has been studied, for example, in \cite{10.2307/2413694}. There has been a rapid growth in the mathematical theory for estimating species trees in recent years, which has been surveyed in \cite{doi:10.1146/annurev-ecolsys-012121-095340}.

The multispecies coalescent model is an extension of the coalescent for a single population to a process on a tree of populations that splits when speciation events occur \cite{steel2016phylogeny}. The species trees in multispecies coalescent models are fixed or drawn from some probability distribution. To show a much wider class of Markov models for trees within trees, nested coalescents were introduced in \cite{blancas2018} to model the joint dynamics of genes and species. The nested coalescent describes the genealogy of a population at the level of both individuals and species. The gene tree describes the coalescence of lineages within a species. The species tree describes the coalescence of different species.

The nested Kingman coalescent is one example of a nested coalescent model, where both the species tree and the gene tree are given by Kingman's coalescent: each pair of species merges independently at some positive constant rate $c$ and each pair of individual lineages belonging to the same species merges independently at rate 1. Recall that one of Kingman's coalescent's striking features is that it descends from infinity. The nested Kingman coalescent also has this feature.  
It was proved in \cite{blancas2019} that the number of total lineages decays as $2\gamma/ct^2$ at time $t$ when $t\to0$, where $\gamma$ is a constant derived from a recursive distributional equation.
In addition to the number of total lineages, the distribution of species mass, which is the number of individual lineages in one species, was studied by solving a McKean-Vlasov equation in \cite{lambert2020coagulation}. The result shows that the typical mass of a species is $O(t^{-1})$ at time $t$ when $t\to0$.

In this article we study the Yule-Kingman nested coalescent model. 
The species tree in the model is a Yule tree. The Yule process is a simple birth process that models a population in which each individual gives birth independently at some constant rate $c$. It is a standard and commonly used model to describe phylogeny at the species level \cite{steel2016phylogeny}. For the Yule-Kingman coalescent, we study the distribution of the number of individual lineages belonging to one species and show the existence and uniqueness of the distribution $\mu_c^*$ that satisfies some recursive distributional equation. We prove that the distribution of the number of individual lineages in each species converges weakly to $\mu_c^*$. Furthermore, the average number of individuals in one species converges in probability to the mean of $\mu_c^*$. When $c=1$, the distribution can be solved explicitly. 

Kingman's coalescent is a special case of the $\Lambda$-coalescent, which is also known as a coalescent with multiple collisions and was introduced by Pitman in \cite{pitman1999} and Sagitov in \cite{10.2307/3215582}. Some results about the Yule-$\Lambda$ nested coalescent model are shown in \cite{YuleLambdaToni}. 
\begin{figure}[ht]
\centering
\captionsetup{width=.85\linewidth} \includegraphics[width = 80mm]{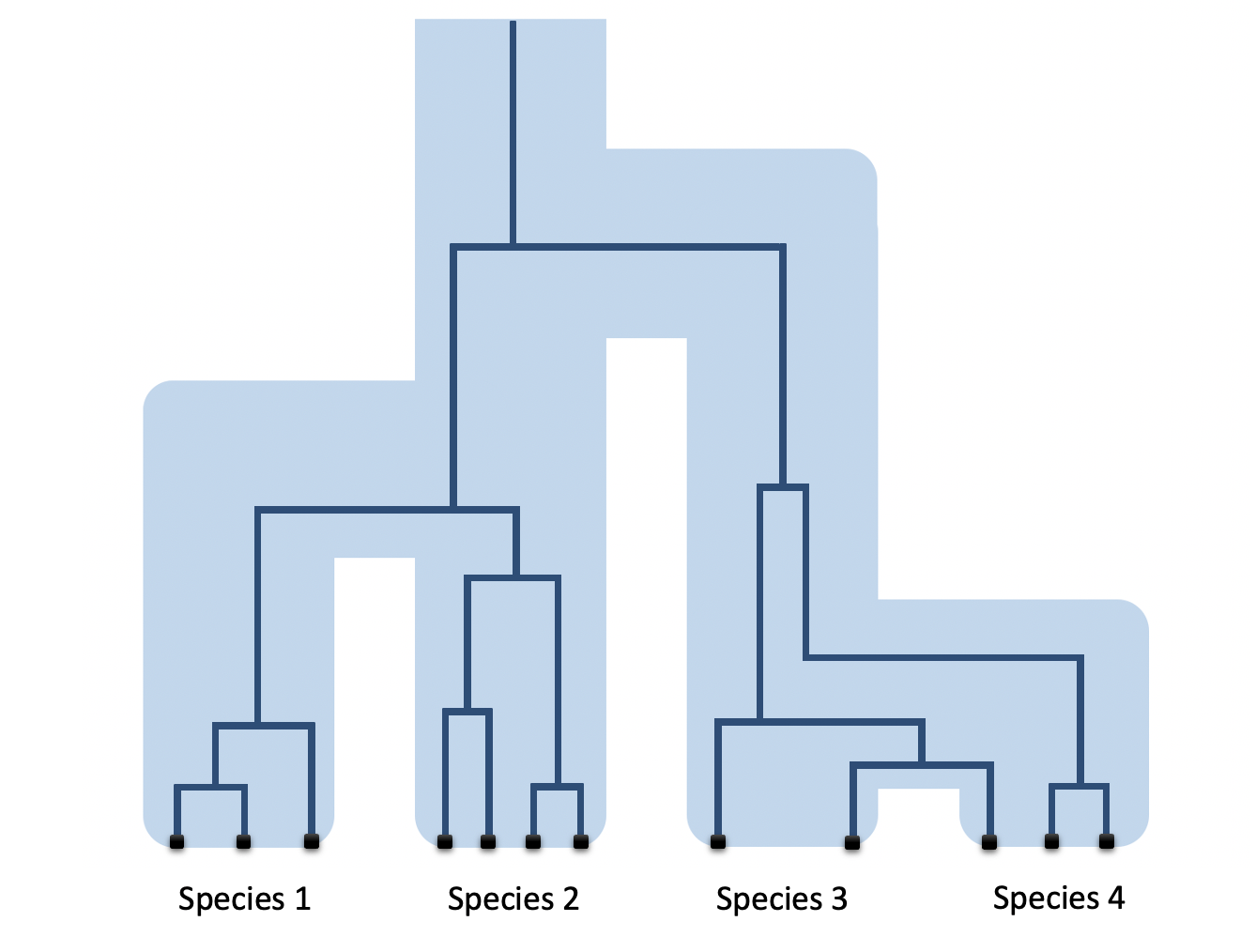}
\caption[A nested coalescent tree which illustrates a model starting from 4 species and 3, 4, 2, and 3 individual lineages in species 1, 2, 3, and 4 respectively.]{A nested coalescent tree which illustrates a model starting from 4 species and 3, 4, 2, and 3 individual lineages in species 1, 2, 3, and 4 respectively. The dark lines present a possible ancestral tree for the sampled individual lineages and the light blue tree in the background shows a possible species tree. Only individual lineages belonging to the same species can merge.}
\label{fig.1}
\end{figure}

The article is organized as follows. We will formally define the model and state the main results in the remainder of this section. Section 2 is devoted to proving the existence and uniqueness of $\mu_c^*$ and section 3 develops other results for the Yule-Kingman nested coalescent.

\subsection{Definition of the Yule-Kingman Nested Coalescent}
The following nested coalescent model, which is similar to the one that appears in \cite{blancas2019}, is considered.
We start with a sample of $s$ species, and $n_k$ individuals are sampled from species $k$ ($n_k$ can be infinite).
Each pair of individuals in a species merges at rate one, and each species has a constant death rate of $c>0$. Once a species dies out, all the individuals of that species will join one of the remaining species with equal probability. We can defined a continuous-time Markov chain to formally describe the process. The Markov chain takes its values in the set of labeled partitions of $\{(m,k)\in\mathbb{Z}\times\mathbb{Z}:1\leqslant m\leqslant n_k,1\leqslant k\leqslant s\}$, in which each block of the partition is labeled with one of the integers $1,\ldots,s$.
Each block represents one lineage, and each label represents one species. At the beginning of the process, there are $\sum_{k=1}^s n_k$ singleton blocks, and block $\{(m,k)\}$ is labeled by $k$. There are two types of transitions:\\

\noindent$\textbf{Lineage mergers:}$ Each pair of blocks with the same label can merge into a single block with rate 1, and the new block will be labeled by the same integer as the two original blocks.\\

\noindent$\textbf{Species mergers:}$ Each label has a constant death rate of $c>0$. If there are no less than two species, once label $j$ dies out, all blocks with label $j$ will have their label changed to $i$, which will be chosen from the remaining labels with equal probability. If there is only one species, once it dies, all blocks vanish simultaneously.\\
 
In a Yule process with birth rate $c>0$, since each branch can generate a new branch with rate $c$, the time between the birth of the $k$th branch and the birth of the $(k+1)$st branch is an exponential random variable with rate $kc$. In our model, since each species has a constant death rate of $c$, the time between the death of the $(k+1)$st to last species and the death of the $k$th to last species is also an exponential random variable with rate $kc$. In a Yule process, the probability for each branch to be the first of the existing branches to generate a new branch is uniform. In our model, the species that dies out will merge with the remaining species with equal probability. Hence if we reverse the direction of the time and redefine the time for the last species in the model to die out as 0, then the evolutionary process of speciation is a Yule process with rate $c$.
Also note that within each species, the lineages merge according to the classical Kingman's coalescent with rate 1. Therefore, we will refer to this model as the nested Yule-Kingman coalescent.

\subsection{Main Results for the Yule-Kingman Nested Coalescent}

Let $S(t)$ be the number of species at time $t$. Notice that the process $(S(t), t\geqslant0)$ is a pure death process with $S(0)=s$. 
Let $N_{k}(0)\geqslant1$ be the number of individuals sampled from species $k$ at time~0, $1\leqslant k\leqslant s$. 
We allow the case in which $N_{k}(0)=\infty$. Each pair of individuals within a species merges at rate 1. Let $N_{k}(t)$ be the number of individuals in species $k$ at time $t$, $1\leqslant k\leqslant s$. Notice that $N_{k}(t)=0$ if species $k$ dies out before time $t$, and there are $S(t)$ species that have a non-zero number of individuals at time $t$.

We use the notation $m_{j}\ll s_{j}$ to mean $\lim_{j\to\infty}m_{j}/s_{j}\to0$. We use $\to_{d}$ and $\to_{p}$ to denote convergence in distribution and convergence in probability respectively. We use $\text{Exp}(c)$ to denote the exponential distribution with rate parameter $c$. We use $[n]$ to denote the set of integers $\{1,\ldots,n\}$. Let $K_{n}(t)$ denote the number of lineages at time $t$ in Kingman's $n$-coalescent. In other words, $K_{n}(t)$ is the number of individual lineages surviving to time $t$ if there are $n$ individuals at time 0 and each pair of individuals merges independently at rate 1. The main results about the nested Yule-Kingman coalescent are stated in the following theorems.

\begin{Theo}\label{existunique}
For any $c>0$, there exists a unique distribution $\mu^*_c$ on $\mathbb{N}^+$ such that independent random variables $W,W_1,W_2$, all with distribution $\mu_c^*$, satisfy the recursive distributional equation 
\begin{equation}\label{First}
W=_{d}K_{W_{1}+W_{2}}(Y),
\end{equation}
where $Y\sim \text{Exp}(c)$ is independent of $W_1$ and $W_2$.
\end{Theo}

\begin{figure}[ht]
\centering
\captionsetup{width=.85\linewidth}
\begin{tikzpicture}[scale=0.8]

\draw [thick,black][](80pt,180pt)--(80pt,80pt);
\draw [thick,black][] (80pt,80pt)--(-20pt,-20pt);
\draw [thick,black][](80pt,80pt)--(180pt,-20pt);
\draw [thick,black][] (30pt,30pt)--(80pt,-20pt);
\draw [thick,black][](165pt,-5pt)--(150pt,-20pt);

\draw [thick, dotted,black][] (80pt,180pt)--(130pt,180pt);
\node[mark size=3pt,color=black] at (240pt,130pt) {};
\draw [thick, dotted,black][] (80pt,80pt)--(130pt,80pt);
\node[mark size=3pt,color=black] at (240pt,80pt) {};
\node[mark size=3pt,color=black] at (0pt,180pt) {$W$};
\node[mark size=3pt,color=black] at (-30pt,50pt) {$W_{1}$};
\node[mark size=3pt,color=black] at (185pt,50pt) {$W_{2}$};

\draw [thick, black, ->] (50pt,180pt) -- (75pt,180pt);
\draw [dash dot, black] (-55pt,198pt) -- (50pt,198pt);
\draw [dash dot, black] (-55pt,162pt) -- (50pt,162pt);
\draw [dash dot, black] (-55pt,198pt) -- (-55pt,162pt);
\draw [dash dot, black] (50pt,198pt) -- (50pt,162pt);

\draw [thick, black, ->] (30pt,50pt) -- (74pt,78pt);
\draw [dash dot, black] (-95pt,68pt) -- (30pt,68pt);
\draw [dash dot, black] (-95pt,32pt) -- (30pt,32pt);
\draw [dash dot, black] (-95pt,68pt) -- (-95pt,32pt);
\draw [dash dot, black] (30pt,68pt) -- (30pt,32pt);

\draw [thick, black, ->] (120pt,50pt) -- (86pt,78pt);
\draw [dash dot, black] (120pt,68pt) -- (250pt,68pt);
\draw [dash dot, black] (120pt,32pt) -- (250pt,32pt);
\draw [dash dot, black] (120pt,68pt) -- (120pt,32pt);
\draw [dash dot, black] (250pt,68pt) -- (250pt,32pt);

\node[mark size=2pt,color=black] at (80pt,180pt) {\pgfuseplotmark{*}};
\node[mark size=2pt,color=black] at (80pt,80pt) {\pgfuseplotmark{*}};
\node[mark size=2pt,color=black] at (30pt,30pt) {\pgfuseplotmark{*}};
\node[mark size=2pt,color=black] at (165pt,-5pt) {\pgfuseplotmark{*}};
\draw [decorate,decoration={brace,amplitude=4pt},xshift=0.5cm,yshift=0pt]
      (120pt,180pt) -- (120pt,80pt) node [midway,right,xshift=.1cm] {$Y\sim \text{Exp}(c)$}; 

\end{tikzpicture}
\caption{Illustration of how the numbers of individual lineages belonging to species evolve in the nested coalescent model.}
\label{fig.2}
\end{figure}
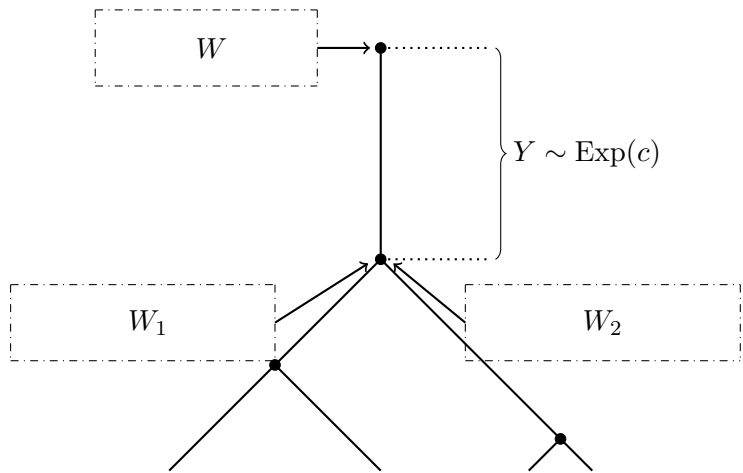

Theorem \ref{existunique} and Figure \ref{fig.2} imply that if the number of individual lineages in each species at each node of the species tree has distribution $\mu^*_c\ast\mu^*_c$, which is the convolution of $\mu^*_c$ with itself, then the number of individuals in each species that is about to merge with another species will have distribution $\mu^*_c$.
From Theorem $\ref{existunique}$ we also have the uniqueness of such a distribution. Section 2 is devoted to proving Theorem \ref{existunique}.

The following theorem illustrates the convergence of the number of individuals in each surviving species to this distribution.

\begin{Theo}\label{Conv}
For a fixed $m$, let $\tau_{m}^{s}=\sup\{t:S(t)=m\}$, which is the first time when the number of species reaches $m-1$ starting from $s$. Denote the left limit $\lim_{t\to(\tau_{m}^{s})^{-}}N_k(t)$ by $N_{k}((\tau_{m}^{s})^{-})$, which represents the number of individuals in species $k$ just before the merger that reduces the number of species from $m$ to $m-1$. Let $\{k_l\}_{l=1}^m\subset[s]$ denote the species that survive to $(\tau_{m}^{s})^{-}$. Then $N_{k_l}((\tau_{m}^{s})^{-})$, $l\in\{1,\ldots,m\}$ are asymptotically independent random variables as $s\to\infty$, and 
\begin{equation*}\label{asymptheo2}
N_{k_l}((\tau_{m}^{s})^{-})\rightarrow_{d}\mu_c^*,~as~s\rightarrow\infty, ~\forall l\in\{1,\ldots,m\}.
\end{equation*}
\end{Theo}

Note that the hypotheses of Theorem \ref{Conv} require $s\to\infty$, but the values of $N_k(0)$, $1\leqslant k\leqslant s$, will not affect the limit as long as they are positive, and they can even be infinity. In addition to the convergence of the number of the individuals in each surviving species, we also have the convergence of the average number of individuals per surviving species, which is stated below.

\begin{Theo}\label{yulekingmantheo31}
Suppose $1\ll m_{j}\ll s_{j}$. Then 
$$\frac{1}{m_{j}}\sum_{k=1}^{s_{j}}N_{k}\left(\left(\tau_{m_{j}}^{s_{j}}\right)^{-}\right)\to_{p} \sum_{i=1}^\infty i\mu_c^*(\{i\}).$$
\end{Theo}

Section 3 is devoted to proving Theorem \ref{Conv} and Theorem \ref{yulekingmantheo31}.
The distribution $\mu_c^*$ for the special case $c=1$ will be presented in section 3.4.

\section{Existence and Uniqueness of the Solution to the RDE (\ref{First})} \label{chap2ENU1.5} 
This section is devoted to proving Theorem $\ref{existunique}$. 
Let $\mathcal{S}$ be the set of probability distributions on $\mathbb{N}^+\cup\{\infty\}$, and let $\mathcal{S}_{1}$ be the set of probability distributions on $\mathbb{N}^+$ with finite mean.
Let $F_c:\mathcal{S}\to\mathcal{S}$ be the mapping defined such that $F_c(\mu)$ is the distribution of 
\begin{equation}\label{dist}
K_{W_{1}+W_{2}}(Y),
\end{equation}
where $W_{1}$ and $W_{2}$ have distribution $\mu$, $Y$ has distribution $\text{Exp}(c)$, and the random variables $Y$, $W_{1}$ and $W_{2}$ are independent. Then the recursive distributional equation $(\ref{First})$ is equivalent to 
\begin{equation}\label{Frecur1}
F_c(\mu_c^*)=\mu_c^*,	
\end{equation}
where $\mu_c^*$ is the distribution of $W$. Let $F_c^{n}: \mathcal{S}\to\mathcal{S}$ be the map obtained by iterating $n$ times the map $F_c$. Let $\delta_1$ denote the unit mass at 1 and let $\delta_{\infty}$ denote the unit mass at $\infty$. Consider the usual stochastic partial order $\preceq$ on $\mathcal{S}$:
$$\mu_1\preceq\mu_2\text{ iff }\mu_1[0,x]\geqslant\mu_2[0,x], ~\forall x\in[0,\infty].$$
We say $\mu_2$ stochastically dominates $\mu_1$ if $\mu_1\preceq\mu_2$. We will use the notation $X_1\preceq X_2$ if the distribution of $X_2$ stochastically dominates the distribution of $X_1$. 
We say $F_c$ is monotone if 
$$\mu_1\preceq\mu_2\text{  implies  } F_c(\mu_1)\preceq F_c(\mu_2).$$

\subsection{Existence of the Solution to the RDE}\label{chap2.1existence}
\begin{Lemma}\label{3mono}
We have $F_c(\mathcal{S}_{1})\subset\mathcal{S}_{1}$, and $F_c^{2}(\mathcal{S})\subset\mathcal{S}_{1}$. Also, $F_c$ is monotone. 
\end{Lemma}
\begin{proof}
Let $\mu\in\mathcal{S}_{1}$, and let $X$ and $\tilde{X}$ be two independent random variables with distribution $\mu$. Let $Y$ be a random variable with distribution $\text{Exp}(c)$. Then $K_{X+\tilde{X}}(Y)$ has distribution $F_c(\mu)$. Therefore,
$$E[K_{X+\tilde{X}}(Y)]\leqslant E[X+\tilde{X}]=2E[X].$$
It follows that $F_c(\mu)\in\mathcal{S}_{1}$, which proves the first statement of the lemma. 

Since the expression in $(\ref{dist})$ increases as $W_1$ and $W_2$ increase, we have that $F_c$ is monotone. Since each distribution in $\mathcal{S}$ is stochastically dominated by $\delta_{\infty}$ and $F_c$ is monotone, if $F_c^{2}(\delta_{\infty})\in\mathcal{S}_{1}$ can be proved, the second statement of the lemma follows. 

Since each pair of individuals merges at rate one, the death rate in Kingman's coalescent is ${i \choose 2}$ when there are $i$ individuals. For $i\leqslant j$ and $i,j\in \mathbb{N}^+\cup\{\infty\}$, the event $\{K_j(Y)=i\}$ is the event that $Y$ is greater than the time it needs to get from $j$ individuals to $i$ individuals, but less than the time it needs to get from $j$ individuals to $i-1$ individuals. Let $Y_l\sim \text{Exp}\big({l\choose 2}\big)$ for all $l\geqslant2$, and let $Y_1=\infty$. Assume the random variables $Y_l$ are independent of $Y$. Then
\begin{align}\label{kji}
	P(K_j(Y)=i)=P\left(\sum_{l=i+1}^{j}Y_l\leqslant Y<\sum_{l=i}^{j}Y_l\right).
\end{align}
Because of the memoryless property of the exponential distribution, we have
\begin{align}\label{jumpp}
	P\left(\sum_{l=i+1}^{j}Y_l\leqslant Y<\sum_{l=i}^{j}Y_l\right)&=P(Y_i>Y)\prod_{l=i+1}^{j}P(Y_l\leqslant Y)\nonumber\\
	&=\begin{cases}
\frac{c}{\binom{i}{2}+c}\prod^{j}_{l=i+1}\frac{\binom{l}{2}}{\binom{l}{2}+c}, &\text{ when } i\geqslant 2, \\
\prod^{j}_{l=2}\frac{\binom{l}{2}}{\binom{l}{2}+c}, &\text{ when } i=1.\\
\end{cases}
\end{align}
Let $X_{1}$ and $\tilde{X}_{1}$ be two independent random variables with distribution $F_c(\delta_{\infty})$, and let $X_{2}$ be a random variable with distribution $F^{2}_c(\delta_{\infty})$. Then by $(\ref{kji})$ and $(\ref{jumpp})$ we have 
\begin{align*}
P(X_{1}=i)=\frac{c}{\binom{i}{2}+c}\prod^{\infty}_{l=i+1}\frac{\binom{l}{2}}{\binom{l}{2}+c}\leqslant\frac{2c}{i(i-1)+2c}\leqslant\frac{2c}{i(i-1)}, ~\forall i\geqslant2,
\end{align*}
and therefore
\begin{align}\label{X1}
P(X_{1}\geqslant i)\leqslant\sum^{\infty}_{l=i}\frac{2c}{l(l-1)}=\frac{2c}{i-1}, ~\forall i\geqslant2.
\end{align}
Hence
\begin{align}\label{2X1}
P(X_{1}+\tilde{X}_{1}\geqslant i)\leqslant 2\cdot P\left(X_{1}\geqslant\frac{i}{2}\right)\leqslant 2\cdot\frac{2c}{i/2-1}=\frac{8c}{i-2}, ~\forall i\geqslant4.
\end{align}
Since 
$$P(X_{2}\geqslant i)=\sum_{j=i}^{\infty}P(X_{1}+\tilde{X}_{1}=j)\cdot P(K_{j}(Y)\geqslant i),$$
and 
\begin{align*}
	P(K_{j}(Y)\geqslant i)\leqslant P(K_{\infty}(Y)\geqslant i)=P(X_{1}\geqslant i),~\forall j\in\mathbb{N}^+\cup\{\infty\},
\end{align*}
we have 
\begin{align}\label{X2}
	P(X_{2}\geqslant i)\leqslant P(X_1\geqslant i)\cdot\sum_{j=i}^{\infty}P(X_{1}+\tilde{X}_{1}=j)=P(X_1\geqslant i)\cdot P(X_{1}+\tilde{X}_{1}\geqslant i).
\end{align}
By $(\ref{X1})$, $(\ref{2X1})$ and $(\ref{X2})$, we have
\begin{equation}\label{Xinfinitynew1}
P(X_{2}\geqslant i)\leqslant P(X_1\geqslant i)\cdot P(X_{1}+\tilde{X}_{1}\geqslant i)\leqslant\frac{2c}{i-1}\cdot\frac{8c}{i-2}=\frac{16c^{2}}{(i-1)(i-2)}, ~\forall i\geqslant4,
\end{equation}
and therefore
\begin{equation*}
E[X_{2}]=\sum_{i=1}^{\infty}P(X_{2}\geqslant i)\leqslant 3+\sum_{i=4}^{\infty}\frac{16c^{2}}{(i-1)(i-2)}<\infty.
\end{equation*}
It follows that $F_c^{2}(\delta_{\infty})\in\mathcal{S}_{1}$, which completes the proof.
\end{proof}

Since $F_c$ is monotone, the sequence of iterates $F_c^n(\delta_1)$ is increasing,  
and therefore the limit 
\begin{equation}\label{defofmu1star}
\mu_1\coloneqq \lim_{n\to\infty} F_c^n(\delta_1)
\end{equation} 
exists in the sense of weak convergence.

Note that for all $k\in\mathbb{N}^+$,
\begin{align}\label{lemma8proof111}
	\left(F_c(\mu_1)\right)(\{k\})&=\sum_{i=k}^\infty \left(\mu_1 \ast \mu_1\right)(\{i\})P(K_i(Y)=k)\nonumber\\
	&= \lim_{n\to\infty}\sum_{i=k}^\infty\left(\left(\mu_1 \ast \mu_1\right)(\{i\})-\left(F_c^{n-1}(\delta_1) \ast F_c^{n-1}(\delta_1)\right)(\{i\})\right)P(K_i(Y)=k)\nonumber\\
	&\quad\quad\quad\quad+\lim_{n\to\infty}\sum_{i=k}^\infty\left(F_c^{n-1}(\delta_1) \ast F_c^{n-1}(\delta_1)\right)(\{i\})P(K_i(Y)=k). 
\end{align}
By Scheffe's theorem and the Dominated Convergence Theorem, the first term on the right-hand side of $(\ref{lemma8proof111})$ is 0. Note that the second term on the right-hand side of $(\ref{lemma8proof111})$ is 
$\mu_1(\{k\})$, which implies that $\mu_1=F_c(\mu_1)$. Then $\mu_1=F_c^2(\mu_1)\preceq F_c^2(\delta_\infty)\in\mathcal{S}_1$. Therefore $\mu_1$ is a distribution on $\mathcal{S}_1$ that satisfies $(\ref{First})$. Note that the same result can be reached by applying Lemma 4 in \cite{aldous2005}.

\subsection{Uniqueness of the Solution to the RDE}
\begin{Lemma}\label{ProbGenFunEqu1}
If $\mu^*_c\in\mathcal{S}_1$ is a distribution such that independent random variables $W,W_1,W_2$, all with distribution $\mu_c^*$, satisfy the recursive distributional equation  
\begin{equation*}\label{FirstinLemma}
W=_{d}K_{W_{1}+W_{2}}(Y),
\end{equation*}
where $Y\sim \text{Exp}(c)$ is independent of $W_1$ and $W_2$, then the probability generating function of $\mu^*_c$, denoted by $R_c$, satisfies
\begin{equation*}\label{Rx}
R_c(x)=R_c^{2}(x)+\frac{x(1-x)}{2c}R_c''(x),~R_c(0)=0,~R_c(1)=1,\text{ and } R_c'(x)>0,~\forall x\in[0,1].
\end{equation*}
\end{Lemma}

\begin{proof}
In Kingman's coalescent, once the number of lineages reaches $i\in\mathbb{N}^+$, the rate at which the number of lineages jumps to $i-1$ is $i(i-1)/2$. Because of the memoryless property
, the probability for the number of lineages to stay at $i$ until time $Y$ is $2c/(i(i-1)+2c)$, and the probability for the number of lineages to reach $i-1$ before time $Y$ is $i(i-1)/(i(i-1)+2c)$. Starting from $W_1+W_2$ lineages, let $B_i$ be the event that the number of lineages reaches $i$ before time $Y$. Then 
\begin{equation*}
P(	K_{W_1+W_2}(Y)=i)=P(B_{i})\frac{2c}{i(i-1)+2c},~\forall i\in\mathbb{N}^+.
\end{equation*}
Note that the number of individuals can reach $i$ before time $Y$ if the number of individuals first reaches $i+1$ and then decreases to $i$ or the number of individuals at the beginning of the coalescent is $i$. It follows that
\begin{equation*}
P(	B_i)=P(B_{i+1})\frac{(i+1)i}{(i+1)i+2c}+P(W_1+W_2=i),~\forall i\in\mathbb{N}^+.
\end{equation*}
Therefore,
\begin{align}\label{probKingman11}
	P(&K_{W_1+W_2}(Y)=i)\nonumber\\
	&=\left(P(B_{i+1})\frac{(i+1)i}{(i+1)i+2c}+P(W_1+W_2=i)\right)\frac{2c}{i(i-1)+2c}\nonumber\\
	&=\left(P(	K_{W_1+W_2}(Y)=i+1)\frac{(i+1)i+2c}{2c}\frac{(i+1)i}{(i+1)i+2c}+P(W_1+W_2=i)\right)\frac{2c}{i(i-1)+2c}\nonumber\\
	&=\frac{(i+1)i}{i(i-1)+2c}P(K_{W_1+W_2}(Y)=i+1)+\frac{2c}{i(i-1)+2c}P(W_1+W_2=i),~\forall i\in\mathbb{N}^+.
\end{align}
Since $\mu^*_c$ is a distribution on $\mathbb{N}^+$, we have $W_1+W_2\geqslant 2$. Therefore, 
\begin{align*}\label{probKingman12}
	P(K_{W_1+W_2}(Y)=1)=\frac{1}{c}P(	K_{W_1+W_2}(Y)=2).
\end{align*}
Since $W=_d K_{W_{1}+W_{2}}(Y)$, it follows from $(\ref{probKingman11})$ that 
\begin{equation}\label{Kxxi12} 
	\left(i(i-1)+2c\right)P(W=i)=(i+1)iP(W=i+1)+2cP(W_1+W_2=i),~\forall i\in\mathbb{N}^+.
\end{equation}
By multiplying both sides of $(\ref{Kxxi12})$ by $x^i$ and summing over $i$, we have
\begin{align}\label{PXn+1}
	\sum_{i=1}^\infty i(i-1)P(W=i)&x^i+2c\sum_{i=1}^{\infty}P(W=i)x^i\nonumber\\
	&=\sum_{i=1}^\infty (i+1)iP(W=i+1)x^{i}+2c\sum_{i=2}^\infty P(W_1+W_2=i)x^i.
\end{align}
Note that $P(W=0)=0$. Let $R_c$ be the probability generating function of $\mu_c^*$. Then 
$$R_c(0)=0,$$
$$\sum_{i=1}^{\infty}P(W=i)x^i=R_c(x),$$
$$\sum_{i=1}^\infty i(i-1) P(W=i)x^i=x^2\sum_{i=1}^\infty i(i-1) P(W=i)x^{i-2}=x^2R_c''(x),$$
$$\sum_{i=1}^\infty(i+1)iP(W=i+1)x^{i}=x\sum_{i=1}^\infty i(i-1) P(W=i)x^{i-2}=xR_c''(x),$$ 
and 
$$\sum_{i=2}^\infty P(W_1+W_2=i)x^i=R_c^2(x).$$
Equation $(\ref{PXn+1})$ can then be transformed to an equality about the probability generating function of $\mu^*_c$: 
$$x^2R_c''(x)+2cR_c(x)=xR_c''(x)+2cR_c^2(x).$$
Note that $\mu_c^*=F_c(\mu_c^*)=F_c^2(\mu_c^*)$. By Lemma \ref{3mono}, we have $\mu_c^*\in\mathcal{S}_1$. It follows that $\mu_c^*(\{\infty\})=0$ and therefore
\begin{equation*}\label{R1}
	R_c(1)=1.
\end{equation*}
Since $\mu_c^*\in\mathcal{S}_1$, there exists some constant $m$ such that $P(W=m)>0$. Then
$$P(W=1)=P(K_{W_1+W_2}=1)\geqslant P(W=m)^2P(K_{2m}(Y)=1).$$
By $(\ref{kji})$ and $(\ref{jumpp})$, we have $P(K_{2m}(Y)=1)>0.$
Therefore,
\begin{equation*}\label{PW1}
P(W=1)\geqslant P(W=m)^2P(K_{2m}(Y)=1)>0.
\end{equation*}
It follows that 
\begin{equation*}\label{R'x0}
R_c'(x)\geqslant P(W=1)>0,~ \forall x\in[0,1].
\end{equation*}
The result of the lemma follows.
\end{proof}

Note that all distributions in $\mathcal{S}$ stochastically dominate $\delta_1$ and are stochastically dominated by $\delta_\infty$. Since $F_c$ is monotone,
 any fixed point of $F_c$ stochastically dominates $\lim_{n\to\infty}F_c^n(\delta_1)$ and is stochastically dominated by $\lim_{n\to\infty}F_c^n(\delta_\infty)$, which implies that if we can prove the limiting distributions $\lim_{n\to\infty}F_c^n(\delta_1)$ and $\lim_{n\to\infty}F_c^n(\delta_\infty)$ are the same, then the fixed point of $F_c$ on $\mathcal{S}$ is unique. That is, the distribution on $\mathbb{N}^+\cup\{\infty\}$ that can satisfy the property in Theorem \ref{existunique} is unique. Recall the definition of $\mu_1$ from (\ref{defofmu1star}). Let
\begin{equation}\label{defofmu2star}
\mu_2\coloneqq \lim_{n\to\infty}F_c^n(\delta_\infty).
\end{equation}
By a similar argument to the one in (\ref{lemma8proof111}), we have $\mu_2=F_c(\mu_2)$. Let $R_1(x)$ and $R_2(x)$ be the probability generating functions of $\mu_1$ and $\mu_2$ respectively.


\begin{Lemma}\label{newwgprime}
For $0\leqslant x\leqslant1$, let 
\begin{equation}\label{defineg}
g_i(x)\coloneqq -\log(1-R_i(x)),~i=1,2.
\end{equation}
Then for $i=1,2$ and for all $x\in(0,1)$, we have
\begin{equation}\label{g}
g_i''(x)=(g_i'(x))^2+2c\frac{1-e^{-g_i(x)}}{x}	\frac{1}{1-x}, ~ g_i(0)=0, ~\lim_{x\to1^-}g_i(x)=\infty,~ 0<g_i(x)<\infty.
\end{equation}
We also have
\begin{equation}\label{Posigprime}
\lim_{x\to1^-}g_i'(x)=\infty,
\end{equation}
\begin{equation}\label{Posigprimeprime}
\lim_{x\to1^-}g_i''(x)=\infty,
\end{equation}
and
\begin{equation}\label{1/1x}
g_i'(x)\leqslant \frac{1}{1-x}, ~\forall x\in[0,1).
\end{equation}
\end{Lemma}
\begin{proof}
It follows from (\ref{defineg}) and Lemma \ref{ProbGenFunEqu1} that for $i=1,2$,
$$g_i(0)=0, ~\lim_{x\to1^-}g_i(x)=\infty,~ 0<g_i(x)<\infty,~\forall x\in(0,1),$$
\begin{equation}\label{Rxg'xg''x}
	g_i'(x)=\frac{R_i'(x)}{1-R_i(x)}>0,~ \forall x\in[0,1), 
\end{equation}
and
\begin{equation}\label{g''proof1}
g_i''(x)=\frac{(R_i'(x))^2}{(1-R_i(x))^2}+\frac{R_i''(x)}{1-R_i(x)}=\frac{(R_i'(x))^2}{(1-R_i(x))^2}+\frac{2cR_i(x)}{x(1-x)}>0,~ \forall x\in[0,1).
\end{equation}
The result in $(\ref{g})$ follows.

By $(\ref{defineg}$) and the definition of the probability generating function, we have
\begin{equation*}
	\frac{1-e^{-g_i(x)}}{x}=\frac{R_i(x)}{x}=\sum_{n=1}^{\infty}\mu_i(\{n\})x^{n-1}\geqslant \mu_i(\{1\})=R_i'(0),~\forall x\in[0,1], ~i=1,2.
\end{equation*}
Recall from Lemma \ref{ProbGenFunEqu1} that $R_i'(x)>0$ for all $x\in[0,1]$, which implies that $(1-e^{-g_i(x)})/x$ is positive and nondecreasing on $[0,1]$. 
Then by $(\ref{g}$), there exists some positive constant $a$ such that 
\begin{equation*}
g_i''(x)>\frac{a}{1-x},~\forall x\in[0,1),~i=1,2 ,
\end{equation*}
which implies $(\ref{Posigprimeprime})$. Equation $(\ref{Posigprime})$ then follows from the Fundamental Theorem of Calculus.
Note that $R_i'(x)=\sum_{n=1}^{\infty}n\mu_i(\{n\})x^{n-1}$ is nondecreasing on $[0,1]$. Then 
\begin{equation*}\label{1R(x)}
R_i'(x)(1-x)\leqslant 1-R_i(x), ~\forall x\in[0,1),~i=1,2.	
\end{equation*}
Therefore,
\begin{equation*}
g_i'(x)=\frac{R_i'(x)}{1-R_i(x)}\leqslant \frac{1}{1-x}, ~\forall x\in[0,1),~i=1,2,
\end{equation*}
which is the result in $(\ref{1/1x})$.
\end{proof}

Since $\delta_1\preceq\delta_\infty$ and $F_c$ is monotone, we have
\begin{equation}\label{stochasmu12}
\mu_1=\lim_{n\to\infty}F_c^n(\delta_1)\preceq\lim_{n\to\infty}F_c^n(\delta_\infty)=\mu_2.
\end{equation}
If $\mu_1\neq \mu_2$, then $R_1(x)>R_2(x)$ and $g_1(x)>g_2(x)$ for all $x\in(0,1)$. Then by ($\ref{g}$), we have 
\begin{equation}\label{g''}
g_1''(x)-(g_1'(x))^2=2c\frac{1-e^{-g_1(x)}}{x}	\frac{1}{1-x}>2c\frac{1-e^{-g_2(x)}}{x}	\frac{1}{1-x}=g_2''(x)-(g_2'(x))^2,~ \forall x\in(0,1).
\end{equation}
Recall from $(\ref{Kxxi12})$ that 
\begin{align*}
\mu_i(\{n\})=\frac{(n+1)n}{n(n-1)+2c}\mu_i(\{n+1\})+\frac{2c}{n(n-1)+2c}\left(\mu_i\ast\mu_i\right)(\{n\}).~\forall n\in\mathbb{N}^+,~i=1,2.
\end{align*}
It follows by an induction argument that if $\mu_1(\{1\})=\mu_2(\{1\})$, then $\mu_1$ and $\mu_2$ must be the same distribution. 
By ($\ref{stochasmu12}$) and our assumption that $\mu_1$ and $\mu_2$ are two different distributions, we have 
$$R_1'(0)=\mu_1(\{1\})>\mu_2(\{1\})=R_2'(0).$$
It follows that
\begin{equation}\label{g'(0)}
g_1'(0)=\frac{R_1'(0)}{1-R_1(0)}>\frac{R_2'(0)}{1-R_2(0)}= g_2'(0).	
\end{equation}
Recall from $(\ref{Rxg'xg''x})$ and ($\ref{g''proof1})$ that $g_i'(x)>0$ and $g_i''(x)>0$ for all $x\in[0,1)$, which implies 
$$g_1'(x)+g_2'(x)>g_1'(0)+g_2'(0)>0,~\forall x\in(0,1).$$ 
By ($\ref{g''}$), ($\ref{g'(0)}$) and the Fundamental Theorem of Calculus, we have 
\begin{align*}\label{g1_g2}
g_1'(x)-g_2'(x)&=	g_1'(0)-g_2'(0)+\int_0^x(g_1''(y)-g_2''(y))dy\\
&>g_1'(0)-g_2'(0)+\int_0^x(g_1'(y))^2-(g_2'(y))^2dy\\
&>\int_0^x(g_1'(y)+g_2'(y))(g_1'(y)-g_2'(y))dy\\
&>(g_1'(0)+g_2'(0))\int_0^x(g_1'(y)-g_2'(y))dy.
\end{align*}
Since $g_1'(0)-g_2'(0)>0$ and $g_1'(x)-g_2'(x)$ is continuous, we have $g_1'(x)-g_2'(x)>0$ for all $x\in[0,1)$. Combining this result with (\ref{Rxg'xg''x}), we have the result that if $\mu_1$ and $\mu_2$ are not identical distributions, then
\begin{equation}\label{g'}
	g_1'(x)>g_2'(x)>0, ~\forall x\in[0,1).
\end{equation}
And it follows from (\ref{g''proof1}), ($\ref{g''}$) and ($\ref{g'}$) that
\begin{equation}\label{g''2}
	g_1''(x)>g_2''(x)>0, ~\forall x\in[0,1).
\end{equation}

\begin{Lemma}\label{lemma7hn}
There exists some constant $\bar{\epsilon}>0$ such that for each $n\geqslant2$, there exists a unique function $h_{1,n}$ defined on $[1-\bar{\epsilon},1-1/n]$ that satisfies the ordinary differential equation 
	\begin{equation}\label{hODE}
	h'_{1,n}(x)=(h_{1,n}(x))^2+\frac{2c}{1-x},~h_{1,n}\left(1-\frac{1}{n}\right)=g_1'\left(1-\frac{1}{n}\right).
	\end{equation}
Also, the limit $\lim_{n\to\infty}h_{1,n}(x)$ exists for all $x\in[1-\bar{\epsilon},1)$. Let 
	$$h_1(x)\coloneqq \lim_{n\to\infty}h_{1,n}(x),~ x\in[1-\bar{\epsilon},1).$$
	Then $h_1$ is a solution to the ordinary differential equation 
\begin{equation}\label{hInitial}
h_1'(x)=(h_1(x))^2+\frac{2c}{1-x},
\end{equation}
with the condition that 
\begin{equation}\label{conditionlemma2}
	\int_{1-\bar{\epsilon}}^x h_1(y)dy<\infty, ~\forall x\in[1-\bar{\epsilon},1),\text{ and } \lim_{x\to 1^-}\int_{1-\bar{\epsilon}}^x h_1(y)dy=\infty.
\end{equation}
\end{Lemma}
\begin{proof}
Firstly we prove the existence and uniqueness of $h_{1,n}$ on the interval stated in the lemma. Let 
	\begin{equation}\label{fderivative}
		f(x,y)\coloneqq y^2+\frac{2c}{1-x}.
	\end{equation}
Then $f(x,y)$ is defined in $[0,1)\times\mathbb{R}\subset\mathbb{R}^2$. In addition, the functions $f$ and $\partial f/\partial y$	are defined and continuous in $[0,1)\times\mathbb{R}$. By Theorem 1.3.1 in \cite{sanchez1979ordinary}, there exists a unique solution $y=h_{1,n}(x)$ of (\ref{hODE}) defined in some neighborhood of $(1-1/n,g'_1\left(1-1/n\right))\in\mathbb{R}^2$. Suppose that $h_{1,n}(x)$ is defined on $(r_{1,n},r_{2,n})$, where $0\leqslant r_{1,n}<1-1/n$ and $1-1/n<r_{2,n}\leqslant 1$.

Let $j_n$ be a function of $x$ that satisfies 
    \begin{equation}\label{jx} 
    j_n'(x)=(j_n(x))^2+4cn, ~j_{n}\left(1-\frac{1}{n}\right)=g'_1\left(1-\frac{1}{n}\right).
    \end{equation}
By Theorem 1.3.1 in \cite{sanchez1979ordinary}, we have that $y=j_n(x)$ is uniquely defined in some neighborhood of $(1-1/n,g'_1\left(1-1/n\right))\in\mathbb{R}^2$. Solving the differential equation, we get
    \begin{equation*}\label{jxsoluton}
    	j_n(x)=\sqrt{4cn}\tan\left(\sqrt{4cn}\left(x+\frac{1}{\sqrt{4cn}}\arctan\left(\frac{g_1'(1-1/n)}{\sqrt{4cn}}\right)-1+\frac{1}{n}\right)\right).
    \end{equation*}
    Therefore, there exists $\epsilon_n>0$ such that 
    $$j_n(x)<\infty,~\forall x\in\left[1-\frac{1}{n},1-\frac{1}{n}+\epsilon_n\right].$$ 
    Note that 
    $$0<\frac{2c}{1-x}\leqslant 4cn,~\forall x\in\left[0,1-\frac{1}{n}+\frac{1}{2n}\right],$$ and 
    $$h_{1,n}\left(1-\frac{1}{n}\right)=j_n\left(1-\frac{1}{n}\right)=g'_1\left(1-\frac{1}{n}\right) >0.$$ 
If $r_{2,n}\leqslant 1-1/n+\left(1/(2n)\wedge\epsilon_n\right)$, then by ($\ref{hODE}$), ($\ref{jx}$) and the comparison theorem, 
    we have
\begin{equation}\label{boundabove}
		h_{1,n}(x)\leqslant j_n(x)<j_n\left(1-\frac{1}{n}+\left(\frac{1}{2n}\wedge\epsilon_n\right)\right)<\infty, \forall x\in \left(1-\frac{1}{n},r_{2,n}\right).
	\end{equation}
Since $f$ is bounded on $(0,1-1/n+\left(1/(2n)\wedge\epsilon_n\right)]\times [0,j_n(1-1/n+\left(1/(2n)\wedge\epsilon_n\right))]$, by Theorem~6.2.1 in \cite{sanchez1979ordinary}, the solution $h_{1,n}(x)$ can be continued to the right of $r_{2,n}$. Therefore, function $h_{1,n}$ exists on $(r_{1,n}, 1-1/n+\left(1/(2n)\wedge\epsilon_n\right)]$.

Recall from (\ref{g'}) that $g_1'(1-1/n)>0$, which implies $h_{1,n}(1-1/n)>0$. If there does not exist $x_n\in(r_{1,n},1-1/n)$ such that 
	\begin{equation}\label{xnsetting}
		h_{1,n}(x)>0, ~\forall x\in\left(x_n, 1-\frac{1}{n}\right] \text{ and } h_{1,n}(x_n)=0,
	\end{equation}
then $h_{1,n}(x)\in(0,g_1'(1-1/n)]$ for all $x\in(r_{1,n},1-1/n]$.
Since the function $f$ is bounded on $(0,1-1/n+\left(1/(2n)\wedge\epsilon_n\right)]\times [0,j_n(1-1/n+\left(1/(2n)\wedge\epsilon_n\right))]$,
by Theorem 6.2.1 in \cite{sanchez1979ordinary}, the solution $h_{1,n}(x)$ can be continued to the left of $r_{1,n}$. Therefore, either there is no such $x_n$ and $h_{1,n}$ exists on $[0,1-1/n+\left(1/(2n)\wedge\epsilon_n\right)]$, or $x_n$ exists on $(0,1-1/n)$ and $h_{1,n}$ exists on $[x_n,1-1/n+\left(1/(2n)\wedge\epsilon_n\right)]$. For the first case, we can set $x_n=0$.

	Recall that $(1-e^{-g_1(x)})/x=\sum_{n=1}^{\infty}\mu_i(\{n\})x^{n-1}$ is a positive and nondecreasing function on $[0,1]$. Then
	\begin{equation}\label{1egxupper}
	\frac{	1-e^{-g_1(x)}}{x}< 1,~\forall x\in(0,1).
	\end{equation}
    If $h_{1,n}(x)$ and $g_1'(x)$ intersect at a point $x^*\in(x_n, 1-1/n)$, then 
	$$h'_{1,n}(x^*)=(h_{1,n}(x^*))^2+\frac{2c}{1-x^*}>(g_1'(x^*))^2+2c\frac{1-e^{-g_1(x^*)}}{x^*}	\frac{1}{1-x^*}=g_1''(x^*).$$ 
	Therefore, by $(\ref{g})$, ($\ref{hODE}$), $(\ref{1egxupper})$ and the comparison theorem, we have $h_{1,n}(x)>g'_1(x)$ for all $x\in(x^*,1-1/n+\left(1/(2n)\wedge\epsilon_n\right)]$. This implies that $h_{1,n}(1-1/n)>g_1'(1-1/n)$, which contradicts ($\ref{hODE}$). Hence we have 
	\begin{equation}\label{difn}
	h_{1,n}(x)<g'_1(x),~ \forall x\in\left(x_n,1-\frac{1}{n}\right).
	\end{equation}
	
	Now we are going to show that there exists $\bar{\epsilon}>0$ such that for all $n\geqslant 2$, we have $x_n<1-1/n-\bar{\epsilon}$.	
	By $(\ref{defineg})$, $(\ref{g})$, ($\ref{hODE}$), $(\ref{xnsetting})$ and $(\ref{difn})$, we have
	\begin{align}
	(g'_1(x)-h_{1,n}(x))'&=(g'_1(x))^2-(h_{1,n}(x))^2+\frac{2c}{1-x}\left(\frac{1-e^{-g_1(x)}}{x}-1\right)\nonumber\\
	&>\frac{2c}{1-x}\left(\frac{1-e^{-g_1(x)}}{x}-1\right)\nonumber\\
	&=2c\frac{R_1(x)-x}{x-x^2},~ \forall x\in\left[x_n, 1-\frac{1}{n}\right)\label{diffg'h}.
	\end{align}
	By L'Hôpital's rule, we have 
	\begin{equation*}\label{Lhospital}
\lim_{x\to1}\frac{R_1(x)-x}{x-x^2}=1-R'_1(1),
\end{equation*}
where $R_1'(1)$ equals the expected value of a random variable with distribution $\mu_1$. By Lemma \ref{3mono}, we have $\mu_1\in\mathcal{S}_1$, which implies $R_1'(1)<\infty$. Note that $(R_1(x)-x)/(x-x^2)>-1/(1-x)$, and $(R_1(x)-x)/(x-x^2)$ is continuous on $(0,1)$. Therefore, there exists some constant $b>0$ such that 
\begin{equation}\label{boundbelow}
	2c\frac{R_1(x)-x}{x-x^2}>-b,~\forall x\in(0,1).
\end{equation}
By $(\ref{hODE})$, $(\ref{diffg'h})$ and $(\ref{boundbelow})$, we have
\begin{equation}\label{g111/n}
g'_1\left(1-\frac{1}{n}-\epsilon\right)-h_{1,n}\left(1-\frac{1}{n}-\epsilon\right)<b\epsilon,~\forall \epsilon\in\left(0,1-\frac{1}{n}-x_n\right].
\end{equation}
Since $g_1''(x)>0$ and $g_1'(x)>0$  on $[0,1)$, by (\ref{g'}) and (\ref{g''2}), we can find an $\bar{\epsilon}\in(0,1/2)$ such that $g'_1(1-1/2-\epsilon)>b\epsilon$ for all $\epsilon\in(0,\bar{\epsilon}]$, and therefore for all $n\geqslant 2$ and $\epsilon\in(0,\bar{\epsilon}]$, we have $g'_1(1-1/n-\epsilon)\geqslant g'_1(1-1/2-\epsilon)>b\epsilon$.   
If $x_n\geqslant 1-1/n-\bar{\epsilon}$, then 
$$g_1'(x_n)\geqslant g_1'\left(1-\frac{1}{n}-\bar{\epsilon}\right)>b\bar{\epsilon}\geqslant b\left(1-\frac{1}{n}-x_n\right).$$ 
Therefore, by $(\ref{g111/n})$, we have 
$$h_{1,n}(x_n)>g_1'(x_n)-b\left(1-\frac{1}{n}-x_n\right)>0,$$ 
which contradicts $(\ref{xnsetting})$. It follows that 
\begin{equation}\label{xnbound}
	x_n< 1-\frac{1}{n}-\bar{\epsilon}.
\end{equation}
Therefore, for all $n\geqslant2$, the function $h_{1,n}(x)$ exists on $[1-\bar{\epsilon},1-1/n+\left(1/(2n)\wedge\epsilon_n\right)]$. Since $f(x,y)=y^2+2c/(1-x)$ is positive on $[0,1)\times\mathbb{R}$, we have that $h_{1,n}(x)$ is an increasing function on $[1-\bar{\epsilon},1-1/n+\left(1/(2n)\wedge\epsilon_n\right)]$, which implies
\begin{equation}\label{upperboundh1n111}
	h_{1,n}(x)\leqslant h_{1,n}\left(1-\frac{1}{n}+\left(\frac{1}{2n}\wedge\epsilon_n\right)\right),~\forall x\in\left[1-\bar{\epsilon},1-\frac{1}{n}+\left(\frac{1}{2n}\wedge\epsilon_n\right)\right].
\end{equation}
By (\ref{boundabove}), $(\ref{xnsetting})$, $(\ref{xnbound})$ and (\ref{upperboundh1n111}), we have
\begin{equation}\label{hbound}
	h_{1,n}(x)\in\left(0,j_n\left(1-\frac{1}{n}+\left(\frac{1}{2n}\wedge\epsilon_n\right)\right)\right],~ \forall x\in\left[1-\bar{\epsilon},1-\frac{1}{n}+\left(\frac{1}{2n}\wedge\epsilon_n\right)\right].
\end{equation}	
Note that $f(x,y)$ is continuous and bounded on the closed rectangle 
$$\bar{S}:~ x\in\left[1-\bar{\epsilon},1-\frac{1}{n}+\left(\frac{1}{2n}\wedge\epsilon_n\right)\right],~ y\in\left[0,j_n\left(1-\frac{1}{n}+\left(\frac{1}{2n}\wedge\epsilon_n\right)\right)\right],$$ 
and $f(x,y)$ satisfies a uniform Lipschitz condition 
$$|f(x,y_1)-f(x,y_2)|\leqslant 2j_n\left(1-\frac{1}{n}+\left(\frac{1}{2n}\wedge\epsilon_n\right)\right)|y_1-y_2|$$ 
on $\bar{S}$. 
The uniqueness of $h_{1,n}(x)$ on $[1-\bar{\epsilon},1-1/n+\left(1/(2n)\wedge\epsilon_n\right)]$ follows from Theorem 10.1 in \cite{agarwal2008introduction}.

Secondly we prove the existence of the limit $h_1(x)$. There exists a positive constant $N$ such that $1-1/n>1-\bar{\epsilon}$ for all $n>N$.  
By ($\ref{hODE}$) and $(\ref{difn})$, we have
\begin{equation}\label{h1n1n+1}
	h_{1,n}\left(1-\frac{1}{n}\right)=g'_1\left(1-\frac{1}{n}\right)>h_{1,n+1}\left(1-\frac{1}{n}\right),~ \forall n>N.
\end{equation}
If there exists a point $\tilde{x}_1\in[1-\bar{\epsilon},1-1/n]$ such that $h_{1,n}(\tilde{x}_1)<h_{1,n+1}(\tilde{x}_1)$, then
$$h'_{1,n}(\tilde{x}_1)=(h_{1,n}(\tilde{x}_1))^2+\frac{2c}{1-\tilde{x}_1}<(h_{1,n+1}(\tilde{x}_1))^2+\frac{2c}{1-\tilde{x}_1}=h'_{1,n+1}(\tilde{x}_1).$$
 By $(\ref{hODE})$ and the comparison theorem, we have $h_{1,n}(x)<h_{1,n+1}(x)$ for all $x\in[\tilde{x}_1,1-1/n]$, which contradicts $(\ref{h1n1n+1})$. Therefore, for $n>N$,
 \begin{equation}\label{comparehnhn1}
 h_{1,n}(x)\geqslant h_{1,n+1}(x),~\forall x\in\left[1-\bar{\epsilon},1-\frac{1}{n}\right].
 \end{equation}
For any fixed $x\in[1-\bar{\epsilon},1)$, there exists $n_x$ large enough such that $x<1-1/n_x$. Since $h_{1,n}(x)>0$ for all $n\geqslant n_x$ and $h_{1,n}(x)$ is monotonically nonincreasing as $n$ increases, we have the existence of $\lim_{n\to\infty}h_{1,n}(x)$ for all $x\in[1-\bar{\epsilon},1)$.

Thirdly we will prove that $h_1(x)$ is a solution to ($\ref{hInitial}$) on $(1-\bar{\epsilon},1)$.
	By $h_1(x)=\lim_{n\to\infty}h_{1,n}(x)$ and $(\ref{hbound})$, we have $h_1(x)\geqslant0$ on $[1-\bar{\epsilon},1)$. By $(\ref{hODE})$ and $(\ref{hbound})$, both $h_{1,n}(x)$ and $h'_{1,n}(x)$ are increasing on $[1-\bar{\epsilon},1)$. For $x\in(1-\bar{\epsilon},1)$ and $\delta\in(0,(x-(1-\bar{\epsilon}))\wedge(1-x))$, since
	\begin{align*}
	\frac{h_1(x+\delta)-h_1(x)}{\delta}	=\lim_{n\to\infty}\frac{h_{1,n}(x+\delta)-h_{1,n}(x)}{\delta}=\lim_{n\to\infty}\frac{1}{\delta}\int_{x}^{x+\delta}h'_{1,n}(y)dy,
	\end{align*}
 and $h'_{1,n}(y)$ is increasing on $(1-\bar{\epsilon},1)$, we have that $(h_1(x+\delta)-h_1(x))/\delta$ is monotonically decreasing as $\delta$ decreases, so the limit $\lim_{\delta\to0^+}(h_1(x+\delta)-h_1(x))/\delta$ exists. 
 Similarly, we have the existence of the limit $\lim_{\delta\to0^+}(h_1(x)-h_1(x-\delta))/\delta$. Since $h'_{1,n}(y)$ is increasing on $(1-\bar{\epsilon},1)$, we have 
 \begin{align*}
 	\frac{h_{1,n}(x)-h_{1,n}(x-\delta)}{\delta}\leqslant h_{1,n}'(x)\leqslant\frac{h_{1,n}(x+\delta)-h_{1,n}(x)}{\delta},
 \end{align*}
 for $x\in(1-\bar{\epsilon},1)$ and $\delta\in(0,(x-(1-\bar{\epsilon}))\wedge(1-x))$.
	Then 
\begin{align*}
		\lim_{\delta\to0^+}\frac{h_1(x+\delta)-h_1(x)}{\delta}&=\lim_{\delta\to0^+}\lim_{n\to\infty}\frac{h_{1,n}(x+\delta)-h_{1,n}(x)}{\delta}\nonumber\\
		&\geqslant\lim_{n\to\infty}h'_{1,n}(x)\nonumber\\
		&=\lim_{n\to\infty}\left(h_{1,n}^2(x)+\frac{2c}{1-x}\right)\nonumber\\
		&=h_{1}^2(x)+\frac{2c}{1-x},
\end{align*}
 and 
\begin{align*}
 \lim_{\delta\to0^+}\frac{h_1(x+\delta)-h_1(x)}{\delta}&=\lim_{\delta\to0^+}\lim_{n\to\infty}\frac{h_{1,n}(x+\delta)-h_{1,n}(x)}{\delta}\nonumber\\
 &\leqslant\lim_{\delta\to0}\lim_{n\to\infty}h'_{1,n}(x+\delta)\nonumber\\
		&=\lim_{\delta\to0}\lim_{n\to\infty}\left(h_{1,n}^2(x+\delta)+\frac{2c}{1-x-\delta}\right)\nonumber\\
		&=\lim_{\delta\to0}h_{1}^2(x+\delta)+\frac{2c}{1-x-\delta}\\
		&=h_{1}^2(x)+\frac{2c}{1-x}.
\end{align*}
Therefore 
\begin{equation}\label{upper}
\lim_{\delta\to0^+}\frac{h_1(x+\delta)-h_1(x)}{\delta}=h_{1}^2(x)+\frac{2c}{1-x}.
\end{equation}
Similarly, we have
\begin{equation}\label{lower}
	\lim_{\delta\to0^+}\frac{h_1(x)-h_1(x-\delta)}{\delta}= h_{1}^2(x)+\frac{2c}{1-x}.
\end{equation}
By $(\ref{upper})$ and $(\ref{lower})$, we have that $h_1'(x)$ exists and $h_1'(x)= h_{1}^2(x)+2c/(1-x)$ on $(1-\bar{\epsilon},1)$.

Finally we prove $h_1(x)$ satisfies the conditions in $(\ref{conditionlemma2})$. Note that for each $x\in(0,1)$, there exists a constant $\bar{N}$ such that $1-1/n>x$ for all $n>\bar{N}$. Therefore, by $(\ref{difn})$, we have $h_{1,n}(x)<g_1'(x)$ for all $n>\bar{N}$, and hence
\begin{equation}\label{h111}
	h_1(x)=\lim_{n\to\infty}h_{1,n}(x)<g_1'(x)<\infty, ~\forall x\in[1-\bar{\epsilon},1).
\end{equation}
Hence $\int_{1-\bar{\epsilon}}^x h_1(y)dy<\infty$ for all $x\in[1-\bar{\epsilon},1)$. 

Now we are going to show that $\lim_{n\to\infty}\int_{1-\bar{\epsilon}}^{1-1/n} h_1(x)dx=\infty$, which completes the proof. By $(\ref{comparehnhn1})$, we have that $g_1'(x)-h_{1,n}(x)$ is nondecreasing as $n$ increases for fixed $x\in(1-\bar{\epsilon},1)$ and large enough $n$. Recall from $(\ref{difn})$ that $g_1'(x)-h_{1,n}(x)>0$ on $(1-\bar{\epsilon},1-1/n)$.
It follows from the Fundamental Theorem of Calculus and the Monotone Convergence Theorem that
 \begin{align}
 	\lim_{n\to\infty}&\left(g_1\left(1-\frac{1}{n}\right)-\int^{1-1/n}_{1-\bar{\epsilon}}h_1(x)dx\right)\nonumber\\
 	&=\int_{1-\bar{\epsilon}}^1\lim_{n\to\infty}\bold{1}_{(1-\bar{\epsilon},1-1/n)}(x)(g_1'(x)-h_1(x))dx+g_1(1-\bar{\epsilon})\nonumber\\
 	&=\int_{1-\bar{\epsilon}}^1\lim_{n\to\infty}\bold{1}_{(1-\bar{\epsilon},1-1/n)}(x)(g_1'(x)-h_{1,n}(x))dx+g_1(1-\bar{\epsilon})\nonumber\\
 	&=\lim_{n\to\infty}\int_{1-\bar{\epsilon}}^{1-1/n}(g_1'(x)-h_{1,n}(x))dx+g_1(1-\bar{\epsilon}).\nonumber
 \end{align}
By $(\ref{g111/n})$ and $(\ref{xnbound})$, we have
$$g'_1(x)-h_{1,n}(x)<b\bar{\epsilon},~\forall x\in\left[1-\bar{\epsilon},1-\frac{1}{n}\right).$$
It follows that
\begin{align*}
	\lim_{n\to\infty}\left(g_1\left(1-\frac{1}{n}\right)-\int^{1-1/n}_{1-\bar{\epsilon}}h_1(x)dx\right)
\leqslant b\bar{\epsilon}^2+g_1(1-\bar{\epsilon})<\infty.
\end{align*}
Recall from $(\ref{Posigprime})$ that $\lim_{n\to\infty}g_1(1-1/n)=\infty$, which implies $\lim_{n\to\infty}\int_{1-\bar{\epsilon}}^{1-1/n} h_1(x)dx=\infty$.
\end{proof}

 \begin{Lemma}\label{equiva1a2lemma}
Suppose $f(x)$ is a differentiable function on $[a_1,a_2]\subseteq\mathbb{R}$ with 
$$|f'(x)|\leqslant k(x)|f(x)|,~\forall x\in[a_1,a_2],$$
where $k(x)$ is integrable on $[a_1,a_2]$. If $f(a_2)=0$ and $\int_{a_1}^{a_2}|k(x)|dx<1$, then $f(x)\equiv0$ on $[a_1,a_2]$.
\end{Lemma}
\begin{proof}
By the Fundamental Theorem of Calculus, we have
\begin{equation*}
f(x)=f(a_2)-\int^{a_2}_{x}f'(y)dy=-\int^{a_2}_{x}f'(y)dy, ~\forall x\in[a_1,a_2].
\end{equation*}
Then
\begin{equation*}
|f(x)|\leqslant\int^{a_2}_{x}|f'(y)|dy \leqslant \int^{a_2}_{x}|k(y)f(y)|dy\leqslant \sup_{a_1\leqslant x\leqslant a_2}|f(x)|\int^{a_2}_{a_1}|k(x)|dx,~\forall x\in[a_1,a_2].
\end{equation*}
Since $f(x)$ is continuous on $[a_1,a_2]$, there exists $x_1\in[a_1,a_2]$ such that $|f(x_1)|=\sup_{a_1\leqslant x\leqslant a_2}|f(x)|$. If $\sup_{a_1\leqslant x\leqslant a_2}|f(x)|>0$, then 
$$|f(x_1)|\leqslant \sup_{a_1\leqslant x\leqslant a_2}|f(x)|\int^{a_2}_{a_1}|k(x)|dx<\sup_{a_1\leqslant x\leqslant a_2}|f(x)|=|f(x_1)|,$$
which is a contradiction. Therefore $\sup_{a_1\leqslant x\leqslant a_2}|f(x)|=0$, which implies $f(x)\equiv0$ on $[a_1,a_2]$.
\end{proof}

\begin{Lemma}\label{lemma9delta1}
There exists a constant $\delta>0$ such that on $[1-\delta,1)$, there exists a unique solution to the ordinary differential equation
\begin{equation}\label{hIn}
h'(x)=(h(x))^2+\frac{2c}{1-x},
\end{equation}
satisfying
\begin{equation}\label{conditionlemma3}
	0<h(x)<\frac{1}{1-x}, ~\forall x\in[1-\delta,1) \text{ and } \lim_{x\to 1^-}\int_{1-\delta}^x h(y)dy=\infty. 
\end{equation}
\end{Lemma}
\begin{proof}
The existence of a solution to $(\ref{hIn})$ satisfying $(\ref{conditionlemma3})$ follows directly from Lemma~\ref{lemma7hn}.

	Now we prove the uniqueness of the solution. Assume that there are two different functions $h_1(x)$ and $h_2(x)$ that are solutions to $(\ref{hIn})$ and both satisfy condition $(\ref{conditionlemma3})$. Then 
	\begin{equation}\label{h'1x}
		h_i'(x)=(h_i(x))^2+\frac{2c}{1-x}> \frac{2c}{1-x}, ~\forall x\in[1-\delta,1),~i=1,2.
	\end{equation}
 If $h_i(x)< -2c\log(1-x)$ for all $x\in(1-\delta,1)$, then $\int_{1-\delta}^1 h_i(y)dy<\infty$, which contradicts $(\ref{conditionlemma3})$, so there exist $\Delta_1,~\Delta_2\in(0,\delta)$ such that 
 \begin{equation}\label{h1x}
 	h_i(1-\Delta_i)> -2c\log(1-(1-\Delta_i)),~ i=1,2.
 \end{equation}
Let $\Delta_3=\Delta_1\wedge\Delta_2$. By $(\ref{h'1x})$ and $(\ref{h1x})$, we have 
\begin{equation*}\label{h1xlowerbound}
	h_i(x)> -2c\log(1-x), ~\forall x\in(1-\Delta_3,1),~ i=1,2.
\end{equation*}
Let 
\begin{equation}\label{cihi}
	c_i(x)\coloneqq \frac{1}{h_i(x)+2c\log(1-x)}+x,~i=1,2.
\end{equation}
Then
\begin{equation}\label{hici}
	h_i(x)=\frac{1}{c_i(x)-x}-2c\log(1-x).
\end{equation}
Since $h_i(x)<1/(1-x)$ on $(1-\Delta_3,1)$ and $-2c\log(1-x)>0$, we have 
\begin{equation}\label{cixlowerbound}
	c_i(x)>1,~\forall x\in(1-\Delta_3,1).
\end{equation}
By $(\ref{hIn})$,  $(\ref{cihi})$ and $(\ref{hici})$, on $(1-\Delta_3,1)$, 
\begin{align}\label{c'x}
c_i'(x)=\frac{-(h_i'(x)-2c/(1-x))}{(h_i(x)+2c\log(1-x))^2}+1&=-h_i(x)^2\cdot\frac{1}{(h_i(x)+2c\log(1-x))^2}+1\nonumber\\
&=-\left(\frac{1}{c_i(x)-x}-2c\log(1-x)\right)^2(c_i(x)-x)^2+1\nonumber\\
&=4c\log(1-x)(c_i(x)-x)-4c^2\log^2(1-x)(c_i(x)-x)^2.
\end{align}
By $(\ref{cixlowerbound})$ and $(\ref{c'x})$, we have $c'_i(x)<0$ on $(1-\Delta_3,1)$. It then follows that the limit $\lim_{x\to1^-}c_i(x)$ exists, and 
$$h_i(x)<\frac{1}{\lim_{x\to1^-}c_i(x)-x}-2c\log(1-x),~\forall x\in(1-\Delta_3,1).$$
If $\lim_{x\to1^-}c_i(x)>1$, then $\int_{1-\Delta_3}^1h_i(x)dx<\infty$, which contradicts $(\ref{conditionlemma3})$. Therefore,
\begin{equation*}\label{limitc}
 \lim_{x\to1^-}c_i(x)=1,~i=1,2.
\end{equation*}
It follows from $(\ref{c'x})$ that on $(1-\Delta_3,1)$,
\begin{align}
&\frac{d(c_2(x)-c_1(x))}{dx}\nonumber\\
&~~=4c\log(1-x)(c_2(x)-c_1(x))-4c^2\log^2(1-x)(c_2(x)-c_1(x))(c_2(x)+c_1(x)-2x).\label{diffder}
\end{align}
Since $\lim_{x\to1^-}c_i(x)=1$ and $c_i(x)>1$ for all $x\in(1-\Delta_3,1)$, there exists $\Delta_4\in(0,\Delta_3)$ such that 
\begin{equation}\label{lessthan4}
	0\leqslant c_2(x)+c_1(x)-2x\leqslant 1,~\forall x\in[1-\Delta_4,1].
\end{equation}
By $(\ref{diffder})$ and $(\ref{lessthan4})$, we have 
\begin{equation*}
	\left|\frac{d(c_2(x)-c_1(x))}{dx}\right|<\left(-4c\log(1-x)+4c^2\log^2(1-x)\right)\left|c_2(x)-c_1(x)\right|, ~\forall x\in[1-\Delta_4,1].
\end{equation*}
Since $\int_0^1(-4c\log(1-x)+4c^2\log^2(1-x))dx$ is bounded, there exists $\Delta_5\in(0,\Delta_4]$ such that 
$$\int_{1-\Delta_5}^1\left(-4c\log(1-x)+4c^2\log^2(1-x)\right)dx<1.$$ 
By Lemma \ref{equiva1a2lemma}, we have $c_1(x)\equiv c_2(x)$ on $[1-\Delta_5,1]$. It follows that 
\begin{align}\label{equivh1h21}
	h_1(x)= h_2(x),~\forall x\in[1-\Delta_5,1).
\end{align} 
Recall the definition of the function $f(x,y)$ from $(\ref{fderivative})$. Since $f(x,y)=y^2+2c/(1-x)$ is continuous and bounded on the closed rectangle 
$$\bar{S}':~x\in\left[1-\delta,1-\frac{\Delta_5}{2}\right],~y\in\left[0,\frac{2}{\Delta_5}\right],$$ 
and $f(x,y)$ satisfies a uniform Lipschitz condition 
$$|f(x,y_1)-f(x,y_2)|\leqslant \frac{4}{\Delta_5}|y_1-y_2|$$ 
on $\bar{S}'$, by Theorem 10.1 in \cite{agarwal2008introduction}, there exists at most one solution on $[1-\delta,1-\Delta_5/2]$ satisfying $h(1-\Delta_5)=h_1(1-\Delta_5)$. It follows that 
\begin{align}\label{equivh1h22}
	h_1(x)\equiv h_2(x),~\forall x\in\left[1-\delta,1-\frac{\Delta_5}{2}\right].
\end{align}
By $(\ref{equivh1h21})$ and $(\ref{equivh1h22})$, we have $h_1(x)\equiv h_2(x)$ on $[1-\delta,1)$. It follows that the solution to $(\ref{hIn})$ with condition $(\ref{conditionlemma3})$ is unique.
\end{proof}

\begin{Lemma}\label{lemma10lalphag}
	Let $\alpha$ be a positive constant. Then there exists $\tilde{\epsilon}>0$ such that for any $\epsilon\in(0,\tilde{\epsilon})$, if there exists a function $l_\epsilon(x)$ on $[1-\epsilon,1)$ that satisfies the ordinary differential equation
	\begin{equation}\label{lODE}
	l''_\epsilon(x)=(l'_\epsilon(x))^2+\frac{2c}{1-x},~ l'_\epsilon(1-\epsilon)=g'_1(1-\epsilon)-\alpha, \text{ and } l_\epsilon(1-\epsilon)<\infty,
	\end{equation}
	then $l_\epsilon(x)<\infty$ for all $x\in[1-\epsilon,1)$ and $\lim_{x\to 1^-}l_\epsilon(x)$ exists and is finite.
\end{Lemma}
\begin{proof}
Recall from $(\ref{defineg})$ that $R_1(x)=1-e^{-g_1(x)}$
 for all $x\in(0,1)$. It then follows from $(\ref{g})$ and ($\ref{lODE}$) that
\begin{align}
g_1''(x)-l_\epsilon''(x)
&=(g_1'(x))^2-(l_\epsilon'(x))^2+2c\frac{1}{1-x}\left(\frac{1-e^{-g_1(x)}}{x}-1\right)\nonumber\\
&=(g_1'(x)-l_\epsilon'(x))(g_1'(x)+l_\epsilon'(x))+2c\frac{R_1(x)-x}{x-x^2}\label{difflg}.
\end{align}
Recall from $(\ref{Posigprime})$ that $\lim_{x\to1^-}g_1'(x)=\infty$, and from ($\ref{boundbelow}$) that $2c(R_1(x)-x)/(x-x^2)>-b$ for all $x\in(0,1)$ for some positive constant $b$. Therefore, there exists 
$\tilde{\epsilon}>0$ such that
\begin{equation}\label{lemma5eq3}
\alpha g_1'(1-\epsilon)+2c\frac{R_1(1-\epsilon)-(1-\epsilon)}{(1-\epsilon)-(1-\epsilon)^2}>0, ~\forall \epsilon\in(0,\tilde{\epsilon}),
\end{equation}
and
\begin{equation}\label{lemma5eqeq}
	g_1'(1-\epsilon)-\alpha>0,~\forall \epsilon\in(0,\tilde{\epsilon}).
\end{equation}
By ($\ref{lODE}$) and ($\ref{lemma5eqeq}$) and the Fundamental Theorem of Calculus, for all $x\in[1-\epsilon,1)$,
\begin{equation}\label{l''l'}
l_\epsilon'(x)=	l_\epsilon'(1-\epsilon)+\int_{1-\epsilon}^xl_\epsilon''(y)dy=g_1'(1-\epsilon)-\alpha+\int_{1-\epsilon}^x\left((l'_\epsilon(y))^2+\frac{2c}{1-y}\right)dy>0.
\end{equation}
By ($\ref{lODE}$), ($\ref{difflg}$), ($\ref{lemma5eq3}$) and ($\ref{lemma5eqeq}$), we have 
\begin{align*}
g_1''(1-\epsilon)&-l_\epsilon''(1-\epsilon)\\
&=(g_1'(1-\epsilon)-l_\epsilon'(1-\epsilon))(g_1'(1-\epsilon)+l_\epsilon'(1-\epsilon))+2c\frac{R_1(1-\epsilon)-(1-\epsilon)}{(1-\epsilon)-(1-\epsilon)^2}\\
&=\alpha g_1'(1-\epsilon)+\alpha (g_1'(1-\epsilon)-\alpha)+2c\frac{R_1(1-\epsilon)-(1-\epsilon)}{(1-\epsilon)-(1-\epsilon)^2}\\
&>0, ~\forall \epsilon\in(0,\tilde{\epsilon}).
\end{align*}

We now show $g_1''(x)>l_\epsilon''(x)$ on $[1-\epsilon,1)$. Recall from $(\ref{lODE})$ that $l_\epsilon'(1-\epsilon)=g'_1(1-\epsilon)-\alpha$. Then 
\begin{equation}\label{sthforget1} 
g_1'(x)-l_\epsilon'(x)=\alpha+\int_{1-\epsilon}^x(g_1''(y)-l_\epsilon''(y))dy>0, ~\forall x\in[1-\epsilon,1).
\end{equation}
Recall from $(\ref{1/1x})$ that $g_1'(x)\leqslant1/(1-x)<\infty$ for all $x\in[0,1)$. It follows that $l'_\epsilon(x)<g_1'(x)<\infty$ for all $x\in[1-\epsilon,1)$. By ($\ref{difflg}$), (\ref{sthforget1}) and the Fundamental Theorem of Calculus, we have
\begin{equation}\label{diffg1l}
g_1''(x)-l_\epsilon''(x)=(g_1'(x)+l_\epsilon'(x))\left(\alpha+\int_{1-\epsilon}^x(g_1''(y)-l_\epsilon''(y))dy\right)+2c\frac{R_1(x)-x}{x-x^2},~\forall x\in[1-\epsilon,1).
\end{equation}
Note that $g_1''(x)-l_\epsilon''(x)$ is continuous on $[1-\epsilon,1)$. We will show $g_1''(x)>l_\epsilon''(x)$ on $[1-\epsilon,1)$ by contradiction. If there exists $x_1\in(1-\epsilon,1)$ such that $g_1''(x_1)-l_\epsilon''(x_1)\leqslant0$, then there exists $x_2\in(1-\epsilon,x_1]$ such that 
\begin{equation}\label{diffg1lx2}
	g_1''(x)-l_\epsilon''(x)>0,~\forall x\in[1-\epsilon,x_2),
\end{equation}
and 
\begin{equation}\label{atx2}
	g_1''(x_2)-l_\epsilon''(x_2)=0.
\end{equation}
By ($\ref{l''l'}$), $(\ref{diffg1l})$ and $(\ref{diffg1lx2})$, we have
$$g_1''(x_2)-l_\epsilon''(x_2)>\alpha(g_1'(x_2)+l_\epsilon'(x_2))+2c\frac{R_1(x_2)-x_2}{x_2-x_2^2}>\alpha g_1'(x_2)+2c\frac{R_1(x)-x_2}{x_2-x_2^2}.$$
The right-hand side is positive because of ($\ref{lemma5eq3}$). It follows that $g_1''(x_2)-l_\epsilon''(x_2)>0$, which contradicts $(\ref{atx2})$. Therefore, we have $g_1''(x)-l_\epsilon''(x)>0$ on $[1-\epsilon,1)$. 

We know from ($\ref{l''l'}$) that $l_{\epsilon}(x)$ is monotonically increasing on $[1-\epsilon,1)$. If $\lim_{x\to 1^-}l_\epsilon(x)$ is $\infty$, then by replacing the $\alpha$ in ($\ref{lODE}$) with another constant $\tilde{\alpha}\in(0,\alpha)$ we can get another solution $\tilde{l}_\epsilon(x)$. It is easy to check that $g_1'(x)>\tilde{l}_\epsilon'(x)>l_\epsilon'(x)$ on $[1-\epsilon,1)$. It follows that $\lim_{x\to 1^-}\tilde{l}_\epsilon(x)=\infty$. Since $g_1'(x)<1/(1-x)$ on $(0,1)$, we have that $l_\epsilon'(x)<1/(1-x)$ and $\tilde{l}_\epsilon'(x)<1/(1-x)$. Therefore $l_\epsilon'(x)$ and $\tilde{l}_\epsilon'(x)$ are two different solutions to $(\ref{hIn})$ satisfying condition $(\ref{conditionlemma3})$, which contradicts Lemma \ref{lemma9delta1}. This implies that $\lim_{x\to 1^-}l_\epsilon(x)$ exists and is finite.
\end{proof}

The proof of the uniqueness of the solution to the recursive distributional function $(\ref{First})$ will be given in section \ref{proofoftheo1}, using the lemmas proved above in this subsection.

\subsection{Proof of Theorem \ref{existunique}}\label{proofoftheo1}
\begin{proof}[Proof of Theorem \ref{existunique}]
	The existence of $\mu_c^*$ has been shown in section \ref{chap2.1existence}. Now we prove the uniqueness of the distribution by contradiction. Suppose $\mu_1$ and $\mu_2$ defined in $(\ref{defofmu1star})$ and $(\ref{defofmu2star})$ are two different distributions. Then their probability generating functions $R_1$ and $R_2$ are different. By Lemma \ref{newwgprime}, we have $g_1(x)=-\log(1-R_1(x))$ and $g_2(x)=-\log(1-R_2(x))$ are two different solutions to $(\ref{g})$ that satisfy $(\ref{1/1x})$. In addition, the functions $g_1$ and $g_2$ satisfy $(\ref{g'})$ and $(\ref{g''2})$. Note that the conclusion of Lemma \ref{lemma7hn} still holds if $g_1$ is replaced by $g_2$. Let $h_{i,n}(x)$ be the function that satisfies the ordinary differential equation with initial condition
	\begin{equation*}\label{hODE2}
	h'_{i,n}(x)=(h_{i,n}(x))^2+\frac{2c}{1-x},~h_{i,n}\left(1-\frac{1}{n}\right)=g_i'\left(1-\frac{1}{n}\right),~i=1,2.
	\end{equation*}
	By $(\ref{1/1x})$, $(\ref{h111})$ and Lemma \ref{lemma7hn}, there exists some constant $\tilde{\epsilon}>0$ such that for $i=1,2$, we have $h_i(x)=\lim_{n\to\infty}h_{i,n}(x)$ exists on $[1-\tilde{\epsilon},1)$, and $h_i$ satisfies 
\begin{equation*}\label{hIn''}
h_i'(x)=(h_i(x))^2+\frac{2c}{1-x},
\end{equation*}
\begin{equation*}\label{conditionlemma6}
	0<h_i(x)<\frac{1}{1-x}, ~\forall x\in(1-\tilde{\epsilon},1),
\end{equation*}
and	
\begin{equation}\label{conditionlemma6'}
 \lim_{x\to1^-}\int_{1-\tilde{\epsilon}}^x h_i(y)dy=\infty,~ i=1,2.
\end{equation}

Now we want to show $h_1$ and $h_2$ are two different functions.	
If there exists an $\alpha>0$ such that $g_1'(x)-h_1(x)>\alpha$ for all $x\in(1-\tilde{\epsilon},1)$, then by Lemma \ref{lemma10lalphag}, we have 
$$\lim_{x\to1^-}\int_{1-\tilde{\epsilon}}^xh_1(y)dy<\infty,$$ 
which contradicts $(\ref{conditionlemma6'})$. It follows that for all $\alpha>0$, there exists some $x\in(1-\tilde{\epsilon},1)$ such that $g_1'(x)-h_1(x)\leqslant \alpha$. By $(\ref{g'})$ and $(\ref{g''2})$, we have 
$$g_1'(x)-g_2'(x)>g_1'(1-\tilde{\epsilon})-g_2'(1-\tilde{\epsilon})>0,~\forall x\in(1-\tilde{\epsilon},1).$$
Let $\alpha=g_1'(1-\tilde{\epsilon})-g_2'(1-\tilde{\epsilon})$. 
Then there exists $x\in(1-\tilde{\epsilon},1)$ such that 
$$g_1'(x)-h_1(x)\leqslant \alpha=g_1'(1-\tilde{\epsilon})-g_2'(1-\tilde{\epsilon})<g_1'(x)-g_2'(x),$$ 
which implies that $h_1(x)>g_2'(x)$. By $(\ref{h111})$, we have $g_2'(x)>h_2(x)$ for all $x\in(1-\tilde{\epsilon},1)$. It follows that $h_1$ and $h_2$ are two different functions. However, this contradicts Lemma \ref{lemma9delta1}, which implies that $\mu_1$ and $\mu_2$ are the same distribution. 

Suppose $\mu_c^*$ is a distribution which satisfies $(\ref{Frecur1})$. Then $\delta_1\preceq\mu_c^*\preceq\delta_\infty$. By the monotonicity of $F_c$, which was proved in Lemma \ref{3mono}, we have 
$$\mu_1=\lim_{n\to\infty}F_c^n(\delta_1)\preceq\lim_{n\to\infty}F_c^n(\mu_c^*)\preceq\lim_{n\to\infty}F_c^n(\delta_\infty)=\mu_2.$$ 
Note that $\lim_{n\to\infty}F^n(\mu_c^*)=\mu_c^*$. Since $\mu_1$ and $\mu_2$ are the same distribution, we have 
\begin{equation*}\label{equalofmu120}
	\mu_1=\mu_c^*=\mu_2,
\end{equation*}
which implies the uniqueness of the solution to $(\ref{Frecur1})$. Since $(\ref{Frecur1})$ and $(\ref{First})$ are equivalent, we have the uniqueness of the solution to the recursive distributional equation $(\ref{First})$. 
\end{proof}

\section{Results on the Yule-Kingman Nested Coalescent}
\subsection{Results on the Yule Tree}\label{sec31}

Consider a Yule tree. Denote the first time when the number of branches reaches $m$ as $u=0$. Let the number of branches at time $u$ be $R(u)$. Then $R(0)=m$ and $R(0^-)=m-1$. If we consider the the portion of the tree after time $u=0$, then we have $m$ subtrees. 

In the following paragraph, we are going to use the same notations for the subtree, branching times and number of lineages as in \cite{blancas2019}. We can place the $m$ subtrees in random order and denote them by $\mathcal{T}^{l,m}$, $l\in\{1,\ldots,m\}$. Let $V^{l,m}$ denote the time when the initial branch splits into two. Let $V^{l,m}_{1}$ and $V^{l,m}_{2}$ denote the times when the two branches created at time $V^{l,m}$ split into two respectively. Given $V^{l,m}_{i_{1}\ldots i_{d}}$, define $V^{l,m}_{i_{1}\ldots i_{d}1}$ and $V^{l,m}_{i_{1}\ldots i_{d}2}$ to be the times when the two branches created at time $V^{l,m}_{i_{1}\ldots i_{d}}$ split again. Let $U^{s+1}_{m}=\inf\{u:R(u)=s+1\}$.

\begin{center}
\setlength{\unitlength}{0.75cm}
\begin{picture}(22,8.5)
\linethickness{0.3mm}
\put(0,1){\line(2,1){11}}
\put(11,6.5){\line(2,-1){11}}
\put(11,6.5){\line(0,1){1.5}}
\put(6,4){\line(2,-1){6}}
\put(17.2,3.4){\line(-2,-1){4.8}}
\put(1,1.5){\line(2,-1){1}}
\put(8,3){\line(-2,-1){4}}
\put(20,2){\line(-2,-1){2}}
\put(11,6.5){\circle*{0.2}}
\put(6,4){\circle*{0.2}}
\put(17.2,3.4){\circle*{0.2}}
\put(1,1.5){\circle*{0.2}}
\put(8,3){\circle*{0.2}}
\put(20,2){\circle*{0.2}}
\put(11,8){\circle*{0.2}}
\put(11.2,8){$u=0$}
\put(11.2,6.5){$V^{l,m}$}
\put(4.5,4){$V_1^{l,m}$}
\put(17.5,3.4){$V_2^{l,m}$}
\put(0,1.5){$V_{11}^{l,m}$}
\put(8.5,3){$V_{12}^{l,m}$}
\put(20.3,2){$V_{22}^{l,m}$}
\end{picture}

\vspace{-.2in}
Figure 3: The tree ${\cal T}^{l,m}$.
\end{center}

\begin{Lemma}\label{lemma12resultonyule}
Fix a positive integer $d$. Then 
$$\lim_{s\to\infty}P\left(U_{m}^{s+1}\geqslant\max\left\{V_{i_{1}\ldots i_{d}}^{l,m}, i_{1},\ldots,i_{d}\in\{1,2\}, 1\leqslant l\leqslant m\right\}\right)=1.$$
\end{Lemma}
\begin{proof}
Since $V_{i_{1}\ldots i_{d}}^{l,m}$ is a sum of $d+1$ independent exponential random variables with parameter $c$, we have $V_{i_{1}\ldots i_{d}}^{l,m}\sim Gamma(d+1,c)$. For a given time $u$, we have
\begin{equation}\label{useintheo5111}
P\left(V_{i_{1}\ldots i_{d}}^{l,m}> u\right)=e^{-cu}\sum_{i=0}^{d}\frac{(cu)^{i}}{i!}.
\end{equation}
Then
\begin{align*}\label{usedintheo311}
P\left(\max\left\{V_{i_{1}\ldots i_{d}}^{l,m}, i_{1},\ldots,i_{d}\in\{1,2\}, 1\leqslant l\leqslant m\right\}\leqslant u\right)&\geqslant1-\sum_{i_{1}\ldots i_{d},l}P\left(V_{i_{1}\ldots i_{d}}^{l,m}>u\right)\\
&=1-m2^{d}e^{-cu}\sum_{i=0}^{d}\frac{(cu)^{i}}{i!}\\
&\to 1~\text{as}~u\to\infty.
\end{align*}
Then for any $\epsilon>0$, there exists $\bar{u}>0$ such that 
\begin{equation}\label{barusection31}
P\left(\max\left\{V_{i_{1}\ldots i_{d}}^{l,m}, i_{1},\ldots,i_{d}\in\{1,2\},1\leqslant l\leqslant m\right\}\leqslant u\right)>1-\frac{\epsilon}{2},~ \forall u>\bar{u}. 
\end{equation}
Since $\{U^{s+1}_{m}< u\}$ is the event that there are more than $s$ species at time $u$, 
and the number of births before time $u$ in a Yule process starting from $m$ species is given by
$$P(R(u)=i|R(0)=m)={i-1 \choose i-m}e^{-cmu}(1-e^{-cu})^{i-m},$$
we have
\begin{equation}\label{usedintheo5222}
P\left(U^{s+1}_{m}<u\right)=e^{-cmu}\sum_{i=s+1}^{\infty}\binom{i-1}{i-m} (1-e^{-cu})^{i-m}.
\end{equation}
It follows that there exists a $\bar{s}\in\mathbb{N}^+$ such that 
\begin{equation}\label{barssection31}
P(U^{s+1}_{m}<\bar{u})<\frac{\epsilon}{2},~\forall s>\bar{s}.
\end{equation}
By $(\ref{barusection31})$ and $(\ref{barssection31})$, for any $\epsilon>0$, there exists $\bar{s}$ such that 
$$P\left(U^{s+1}_{m}\geqslant\max\left\{V_{i_{1}\ldots i_{d}}^{l,m}, i_{1},\ldots,i_{d}\in\{1,2\},1\leqslant l\leqslant m\right\}\right)>1-\epsilon,~\forall s>\bar{s}.$$ 
The result of the lemma follows.
\end{proof}

\subsection{Proof of Theorem \ref{Conv}}

If we follow the Yule tree in reverse time from time $u=U^{s+1}_{m}$ down to time $u=0$, we can get a pure death process starting from $s$ species at time $t=0$. The time for the number of species to reach $m-1$ will be $\tau_{m}^{s}$, which is defined in Theorem \ref{Conv}, and we have
\begin{equation}\label{Taums} 
\tau_{m}^{s}=_d U^{s+1}_{m}.
\end{equation}
Note that the time $u=0$ corresponds to the time $t=(\tau_m^s)^-$.

\begin{proof}[Proof of Theorem \ref{Conv}]
Recall that $\{F^n_c(\delta_1)\}_{n=1}^\infty$ is an increasing sequence of distributions which converges weakly to $\mu_c^*$ and $\{F^n_c(\delta_\infty)\}_{n=1}^\infty$ is a decreasing sequence of distributions which converges weakly to $\mu_c^*$. Then for every $x>0$, for any $\epsilon>0$, there exists $\bar{d}\in\mathbb{N}^+$ such that \begin{equation*}
0\leqslant\left(F^{d+1}_c(\delta_1)\right)([0,x])-\mu_c^*([0,x])<\frac{\epsilon}{2}, ~\forall d\geqslant\bar{d}, 
\end{equation*}
and
\begin{equation*}
0\geqslant \left(F^{d+1}_c(\delta_\infty)\right)([0,x])-\mu_c^*([0,x])>-\frac{\epsilon}{2}, ~\forall d\geqslant\bar{d}, 
\end{equation*}
By Lemma \ref{lemma12resultonyule} and $(\ref{Taums})$, for any $\bar{d}\in\mathbb{N}^+$, there exists $\bar{s}\in\mathbb{N}^+$ such that 
\begin{equation*}\label{T12theo5}
P\left(\tau_{m}^{s}\geqslant\max\left\{V_{i_{1}\ldots i_{\bar{d}}}^{l,m}, i_{1},\ldots,i_{\bar{d}}\in\{1,2\}, 1\leqslant l\leqslant m\right\}\right)>1-\frac{\epsilon}{2},~ \forall s\geqslant \bar{s}.
\end{equation*}
Let 
$$A\coloneqq\left\{\tau_{m}^{s}\geqslant\max\left\{V_{i_{1}\ldots i_{\bar{d}}}^{l,m}, i_{1},\ldots,i_{\bar{d}}\in\{1,2\}, 1\leqslant l\leqslant m\right\}\right\}.$$ 
Then $P\left(A^C\right)<\epsilon/2$.
Since the number of individual lineages belonging to each species at any time is at least 1 and at most infinity, we can lower and upper bound the number of individual lineages in each species immediately below any branchpoint at time $V_{i_{1}\ldots i_{\bar{d}}}^{l,m}$, where $i_{1},\ldots,i_{\bar{d}}\in\{1,2\}$ and $1\leqslant l\leqslant m$, by 1 and $\infty$ respectively. It follows that the number of individual lineages in each species immediately below any branchpoint at time $V_{i_{1}\ldots i_{\bar{d}-i}}^{l,m}$ can be lower bounded by some random variable with distribution $F^i_c(\delta_1)$ and upper bounded by some random variable with distribution $F^i_c(\delta_\infty)$. Note that on event $A$, all times $V_{i_{1}\ldots i_{\bar{d}}}^{l,m},i_{1},\ldots,i_{\bar{d}}\in\{1,2\},1\leqslant l\leqslant m$ are less than $\tau_{m}^{s}$. Therefore, there exist independent random variables $X_l^{\bar{d}+1}$, $l\in\{1,\ldots,m\}$ with distribution $F^{\bar{d}+1}_c(\delta_1)$ and independent random variables $Y_l^{\bar{d}+1}$, $l\in\{1,\ldots,m\}$ with distribution $F^{\bar{d}+1}_c(\delta_\infty)$ such that $X_l^{\bar{d}+1}\leqslant N_{k_l}\left(\left(\tau_{m}^{s}\right)^-\right)\leqslant Y_l^{\bar{d}+1}$ for all $l\in\{1,\ldots,m\}$ on event $A$. It follows that for any $x>0$ and $l\in\{1,\ldots,m\}$, we have
$$P\left(\left\{N_{k_l}\left(\left(\tau_{m}^{s}\right)^-\right)\leqslant x\right\}\cap A\right)\leqslant P\left(\left\{X_l^{\bar{d}+1}\leqslant x\right\}\cap A\right)\leqslant P\left(X_l^{\bar{d}+1}\leqslant x\right)= \left(F^{\bar{d}+1}_c(\delta_1)\right)([0,x]),$$
and
\begin{align*}
	P\left(\left\{N_{k_l}\left(\left(\tau_{m}^{s}\right)^-\right)\leqslant x\right\}\cap A\right)&\geqslant P\left(\left\{Y_l^{\bar{d}+1}\leqslant x\right\}\cap A\right)\\
	&\geqslant P\left(Y_l^{\bar{d}+1}\leqslant x\right)-P\left(A^C\right)\\
	&= \left(F^{\bar{d}+1}_c(\delta_\infty)\right)([0,x])-P\left(A^C\right).
\end{align*}
Then
\begin{align}\label{prooftheo2eq101}
 P\left(N_{k_l}\left(\left(\tau_{m}^{s}\right)^-\right)\leqslant x\right)-\mu_c^*([0,x])&\leqslant  P\left(\left\{N_{k_l}\left(\left(\tau_{m}^{s}\right)^-\right)\leqslant x\right\}\cap A\right)+P\left(A^C\right)-\mu_c^*([0,x])\nonumber\\
&\leqslant \left(F^{\bar{d}+1}_c(\delta_1)\right)([0,x])+P\left(A^C\right)-\mu_c^*([0,x])\nonumber\\
&<\epsilon, ~ \forall s\geqslant\bar{s}, ~\forall l\in\{1,\ldots,m\},
\end{align}
and
\begin{align}\label{prooftheo2eq202}
 P\left(N_{k_l}\left(\left(\tau_{m}^{s}\right)^-\right)\leqslant x\right)-\mu_c^*([0,x])& \geqslant  P\left(\left\{N_{k_l}\left(\left(\tau_{m}^{s}\right)^-\right)\leqslant x\right\}\cap A\right)-\mu_c^*([0,x])\nonumber\\
&\geqslant  \left(F^{\bar{d}+1}_c(\delta_\infty)\right)([0,x])-P\left(A^C\right)-\mu_c^*([0,x])\nonumber\\
& >-\epsilon, ~ \forall s\geqslant\bar{s}, ~\forall l\in\{1,\ldots,m\}.
\end{align}
The result 
in Theorem \ref{Conv} follows from (\ref{prooftheo2eq101}) and (\ref{prooftheo2eq202}).

It remains to show the asymptotic independence.  
Let $\{a_1,\ldots,a_m\}$ be an arbitrary point of $\left(\mathbb{N}^+\right)^m$. Then
\begin{align*}
	&P\left(N_{k_{l}}\left(\left(\tau_{m}^{s}\right)^-\right)\leqslant a_l,\forall l\in\{1,\ldots,m\}\right)\nonumber\\
	&\quad\quad\leqslant P\left(\left\{N_{k_{l}}\left(\left(\tau_{m}^{s}\right)^-\right)\leqslant a_l,\forall l\in\{1,\ldots,m\}\right\}\cap A\right)+P\left(A^C\right)\nonumber\\
	&\quad\quad\leqslant P\left(\left\{X^{\bar{d}+1}_l\leqslant a_l,\forall l\in\{1,\ldots,m\}\right\}\cap A\right)+P\left(A^C\right)\nonumber\\
	&\quad\quad\leqslant P\left(X^{\bar{d}+1}_l\leqslant a_l,\forall l\in\{1,\ldots,m\}\right)+P\left(A^C\right)\nonumber\\
	&\quad\quad\leqslant \prod_{l=1}^m \left(F^{\bar{d}+1}_c(\delta_1)\right)([0,a_l])+P(A^C)\nonumber\\
	&\quad\quad\leqslant \prod_{l=1}^m\left(\mu_c^*([0,a_l])+\frac{\epsilon}{2}\right)+\frac{\epsilon}{2}, ~ \forall s\geqslant\bar{s},
\end{align*}
and
\begin{align*}
	P\left(N_{k_{l}}\left(\left(\tau_{m}^{s}\right)^-\right)\leqslant a_l,\forall l\in\{1,\ldots,m\}\right)&\geqslant P\left(\left\{N_{k_{l}}\left(\left(\tau_{m}^{s}\right)^-\right)\leqslant a_i,\forall l\in\{1,\ldots,m\}\right\}\cap A\right)\nonumber\\
	&\geqslant P\left(\left\{Y^{\bar{d}+1}_l\leqslant a_l,\forall l\in\{1,\ldots,m\}\right\}\cap A\right)\nonumber\\
	&\geqslant P\left(Y^{\bar{d}+1}_l\leqslant a_l,\forall l\in\{1,\ldots,m\}\right)-P\left(A^C\right)\nonumber\\
	&\geqslant \prod_{l=1}^m \left(F^{\bar{d}+1}_c(\delta_\infty)\right)([0,a_l])-P(A^C)\nonumber\\
	&\geqslant \prod_{l=1}^m\left(\mu_c^*([0,a_l])-\frac{\epsilon}{2}\right)-\frac{\epsilon}{2}, ~ \forall s\geqslant\bar{s}.
\end{align*}
Since $\epsilon$ can be arbitrarily small, we have
$$\lim_{s\to\infty}P\left(N_{k_{l}}\left(\left(\tau_{m}^{s}\right)^-\right)\leqslant a_l,~\forall l\in\{1,\ldots,m\}\right)= \prod_{l=1}^m\mu_c^*([0,a_l]),$$
which implies that $N_{k_l}((\tau_{m}^{s})^{-})$, $l\in\{1,\ldots,m\}$ are asymptotically independent random variables as $s\to\infty$.
\end{proof}

\subsection{Proof of Theorem \ref{yulekingmantheo31}}
\begin{proof}[Proof of Theorem \ref{yulekingmantheo31}]
	Note that $\tau_s^s$ is the time of the death of the first of the $s$ species and $\tau_{i}^s-\tau_{i+1}^s$ is the time between the death of the $(i+1)$st to last species and the death of the $i$th to last species. It follows that $\tau_s^s\sim\text{Exp}(sc)$, $\tau_{i}^s-\tau_{i+1}^s\sim\text{Exp}(ic)$, and $\tau_s^s$, $\tau_{i}^s-\tau_{i+1}^s$, $i\in\{1,\ldots,s-1\}$ are independent random variables. Therefore,
	$$E[\tau_{m_j}^{s_j}]=\sum_{i=m_j}^{s_j} \frac{1}{ic}\to\infty,~\text{as }j\to\infty,$$
	and 
	$$Var(\tau_{m_j}^{s_j})=\sum_{i=m_j}^{s_j} \frac{1}{(ic)^2}<\sum_{i=1}^{\infty} \frac{1}{i^2c^2}.$$
It follows from Chebyshev's Inequality that for any $\bar{u}>0$, $\epsilon>0$, there exists $J$ such that 
\begin{equation}\label{jJtheo3111}
P(\tau_{m_j}^{s_j}>\bar{u})>1-\frac{\epsilon}{2},~\forall j>J.
\end{equation}
Let $\{k_l\}_{l=1}^{m_j}\subset[s_j]$ denote the species that survive to $(\tau_{m_j}^{s_j})^{-}$. On the event $\left\{\tau_{m_j}^{s_j}>\bar{u}\right\}$, consider one of the $m_j$ subtrees, denoted by $\mathcal{T}^{l, m_j}$. Denote the number of species in the subtree at time $u=\bar{u}$ as $S_{\bar{u}}$. Since each species gives birth to a new species at rate $c$, the random variable $S_{\bar{u}}$ is geometrically distributed with mean $e^{c\bar{u}}$. Recall from section 1 
that $K_{\infty}(t)$ is the number of lineages at time $t$ in Kingman's coalescent and $E[K_{\infty}(t)]$ is finite for all $t>0$. Since the number of individual lineages in each species at time $u=\bar{u}$ is at most $\infty$, on the event $\left\{\tau_{m_j}^{s_j}>\bar{u}\right\}$, we can upper bound $N_{k_{l}}\left(\left(\tau_{m_j}^{s_j}\right)^-\right)$ by $S_{\bar{u}}K_{\infty}(\bar{u})$ for all $l\in\{1,\ldots,m\}$, 
where $S_{\bar{u}}K_{\infty}(\bar{u})$ is a random variable with finite mean.
 Note that the $m_j$ subtrees between time $u=0$ and $u=\bar{u}$ are independent. Then there exist independent random variables $W_l$, $l\in\{1,\ldots,m_j\}$ with the same distribution as $S_{\bar{u}}K_{\infty}(\bar{u})$ such that  
\begin{equation*}
N_{k_{l}}\left(\left(\tau_{m_j}^{s_j}\right)^-\right)\leqslant 	W_l,~\forall l\in\{1,\ldots,m_j\}.
\end{equation*}

By (\ref{useintheo5111}), we have
\begin{equation*}\label{maxvjtheo311}
P\left(\max\left\{V_{i_{1}\ldots i_{d}}^{l,m_j}, i_{1},\ldots,i_{d}\in\{1,2\}\right\}\leqslant u\right)\geqslant 1-2^{d}e^{-cu}\sum_{i=0}^{d}\frac{(cu)^{i}}{i!}\to 1~\text{as}~u\to\infty,
\end{equation*}
which implies that for any $d>0$, $\epsilon>0$ and $n>0$, there exists $\bar{u}>0$ such that
\begin{equation*}\label{maxvjtheo311'}
P\left(\max\left\{V_{i_{1}\ldots i_{d}}^{l,m_j}, i_{1},\ldots,i_{d}\in\{1,2\}\right\}> \bar{u}\right)< \frac{\epsilon}{2n}.
\end{equation*}
Let $A_{l,j,d}\coloneqq \left\{\max\left\{V_{i_{1}\ldots i_{d}}^{l,m_j}, i_{1},\ldots,i_{d}\in\{1,2\}\right\}>\bar{u}\right\}$. Then $A_{l,j,d}$ are independent events for all $l\in\{1,\ldots,m_j\}$. Let $K$ denote the number of values of $l$ such that the event $A_{l,j,d}$ occurs. Then by Markov's inequality,
\begin{equation}\label{jJtheo3222}
P\left(K>\frac{m_j}{n}\right)<\frac{P(A_{l,j,d})m_j}{m_j/n}<\frac{\epsilon}{2}.
\end{equation}
Let $S\subset\{1,\ldots,m_j\}$ denote the set of the values of $l$ such that the event $A_{l,j,d}$ occurs. Denote the cardinality of $S$ as $|S|$. Then $K=|S|$. 
Since the number of individual lineages belonging to each species at any time is at least 1 and at most $\infty$, there exists a random variable $X_l^{d+1}$ with distribution $F^{d+1}_c(\delta_1)$ and a random variable $Y_l^{d+1}$ with distribution $F^{d+1}_c(\delta_\infty)$ such that on the event $A_{l,j,d}^C$, we have $X_l^{d+1}\leqslant N_{k_l}\left(\left(\tau_{m_j}^{s_j}\right)^-\right)\leqslant Y_l^{d+1}$. 
Since the subtrees between time $u=0$ and $u=\bar{u}$ are independent, there exist independent random variables $X_l^{d+1}$, $l\in\{1,\ldots,m_j\}$ and independent random variables $Y_l^{d+1}$, $l\in\{1,\ldots,m_j\}$ such that on $\left\{\tau_{m_j}^{s_j}>\bar{u}\right\}\cup\left\{K>m_j/n\right\}$,
\begin{equation}\label{theo3eq811}
\frac{1}{m_{j}}\sum_{k=1}^{s_{j}}N_{k}\left(\left(\tau_{m_{j}}^{s_{j}}\right)^{-}\right)\geqslant \frac{1}{m_{j}}\sum_{l\notin S}X_l^{d+1},
\end{equation}
and
\begin{equation}\label{theo3eq822} 
\frac{1}{m_{j}}\sum_{k=1}^{s_{j}}N_{k}\left(\left(\tau_{m_{j}}^{s_{j}}\right)^{-}\right)\leqslant \frac{1}{m_{j}}\left(\sum_{l=1}^{m_j}Y_l^{d+1}+\sum_{l\in S}W_l\right).
\end{equation}
By the weak law of large numbers, we have
$$\frac{1}{m_{j}}\sum_{l=1}^{m_j}X_l^{d+1}\to_p E\left[X_l^{d+1}\right],\text{ as }j\to\infty,$$
and
$$\frac{1}{m_{j}}\sum_{l=1}^{m_j}Y_l^{d+1}\to_p E\left[Y_l^{d+1}\right],\text{ as }j\to\infty.$$
Since $P(K>m_j/n)<\epsilon/2$ and $K=|S|$, by Markov's inequality, we have
\begin{align*}
	P\left(\frac{1}{m_{j}}\sum_{l\in S}W_l>\epsilon\right)&\leqslant P\left(\left\{\frac{K}{m_{j}}\frac{1}{|S|}\sum_{l\in S}W_l >\epsilon\right\}\cap\left\{K\leqslant \frac{m_j}{n}\right\}\right)+P\left(K> \frac{m_j}{n}\right)\\
	&< P\left(\frac{1}{n}\frac{1}{|S|}\sum_{l\in S}W_l>\epsilon\right)+\frac{\epsilon}{2}\\
	&\leqslant \frac{E\left[\frac{1}{n}\frac{1}{|S|}\sum_{l\in S}W_l\right]}{\epsilon}+\frac{\epsilon}{2}.
\end{align*}
Recall that $W_l,l\in\{1,\ldots,m\}$ have the same distribution as $S_{\bar{u}}K_{\infty}(\bar{u})$, which is positive and has finite mean. Then for any given $\epsilon>0$, by choosing sufficiently large $n$, the right-hand side of the inequality above can be less than $\epsilon$, which implies that 
$$\frac{1}{m_{j}}\sum_{l\in S}W_l\to_p0,\text{ as }j\to\infty.$$
Since $X_l^{d+1},l\in\{1,\ldots,m\}$ are positive and have finite mean, by a similar argument to that used above, we have
$$\frac{1}{m_{j}}\sum_{l\in S}X_l^{d+1}\to_p0,\text{ as }j\to\infty.$$
Therefore,
\begin{equation}\label{theo3eq833}
\frac{1}{m_{j}}\sum_{l\notin S}X_l^{d+1}=\frac{1}{m_{j}}\sum_{l=1}^{m_j}X_l^{d+1}-\frac{1}{m_{j}}\sum_{l\in S}X_l^{d+1}\to_p E\left[X_l^{d+1}\right],\text{ as }j\to\infty,
\end{equation}
and
\begin{equation}\label{theo3eq844}
\frac{1}{m_{j}}\left(\sum_{l=1}^{m_j}Y_l^{d+1}+\sum_{l\in S}W_l\right)=\frac{1}{m_{j}}\sum_{l=1}^{m_j}Y_l^{d+1}+\frac{1}{m_{j}}\sum_{l\in S}W_l\to_p E\left[Y_l^{d+1}\right],\text{ as }j\to\infty.
\end{equation}
Recall from (\ref{Xinfinitynew1}) that
$$\left(F^2_c(\delta_\infty)\right)([i,\infty))\leqslant \frac{16c^{2}}{(i-1)(i-2)}, ~\forall i\geqslant4.$$
Since $\{F^d_c(\delta_\infty)\}_{d=2}^\infty$ is a decreasing sequence of distributions, we have
$$\lim_{a\to\infty}\sup_{d\geqslant2} E\left[\left|Y_l^d\right|\bold{1}_{\{\left|Y_l^d\right|\geqslant a\}}\right]=\lim_{a\to\infty}E\left[\left|Y_l^2\right|\bold{1}_{\{\left|Y_l^2\right|\geqslant a\}}\right] =0,$$
which implies $Y_l^d$ are uniformly integrable. Since $F^d_c(\delta_\infty)\to_p \mu_c^*$, by Theorem 25.12 in \cite{billingsley1995probability}, we have 
\begin{equation}\label{theo3eq855}
\lim_{d\to\infty} E\left[Y_l^{d+1}\right]=\sum_{i=1}^\infty i\mu_c^*(\{i\}).
\end{equation}
Similarly, we have
\begin{equation}\label{theo3eq866}
\lim_{d\to\infty} E\left[X_l^{d+1}\right]=\sum_{i=1}^\infty i\mu_c^*(\{i\}).
\end{equation}
By (\ref{theo3eq811}), (\ref{theo3eq822}), (\ref{theo3eq833}), (\ref{theo3eq844}), (\ref{theo3eq855}) and (\ref{theo3eq866}), we have
\begin{align*}
	\lim_{j\to\infty}P\left(\left\{\left|\frac{1}{m_{j}}\sum_{k=1}^{s_{j}}N_{k}\left(\left(\tau_{m_{j}}^{s_{j}}\right)^{-}\right)-\sum_{i=1}^\infty i\mu_c^*(\{i\})\right|>\epsilon\right\}\cap \left(\left\{\tau_{m_j}^{s_j}>\bar{u}\right\}\cup\left\{K>\frac{m_j}{n}\right\}\right)\right)=0.
\end{align*}
By (\ref{jJtheo3111}) and (\ref{jJtheo3222}), we have 
$$P\left(\left(\left\{\tau_{m_j}^{s_j}>\bar{u}\right\}\cap\left\{K>\frac{m_j}{n}\right\}\right)^C\right)<\epsilon,~\forall j>J.$$
Since $\epsilon$ can be arbitrarily small, the result follows.	
\end{proof}

\subsection{An Example of $\mu_c^*$}
\begin{Theo}
When $c=1$, we have the unique distribution $\mu_c^*$ in Theorem \ref{existunique},
$$\mu_1^*(\{i\})=\frac{2i-1}{3^{i}}, \forall i\in\mathbb{N}^+.$$
\end{Theo}
\begin{proof}
When $c=1$, it follows from Lemma \ref{ProbGenFunEqu1} that the probability generating function of $\mu_1^*$, denoted by $R$, satisfies 
$$R(x)=R(x)^2+\frac{x}{2}R''(x)-\frac{x^2}{2}R''(x),~\forall x\in[0,1],~ R(0)=0,~\text{and } R(1)=1. $$
It is easy to check that 
$$R(x)=\frac{x^2+3x}{(x-3)^2}$$
is a solution to the ODE. Note that 
$$\frac{x^2+3x}{(x-3)^2}=\sum_{i=1}^{\infty}\left(\frac{2i-1}{3^{i}}\right)x^{i}.$$
Let $W_1$ and $W_2$ be two independent random variables with $P(W_n=i)=(2i-1)/3^i$. Then $W_1+W_2$ has probability generating function $R^2(x)=(x^2+3x)^2/(x-3)^4$, which implies 
$$P(W_1+W_2=i)=\frac{2i^3-6i^2+7i-3}{3^{1+i}}.$$
Recall from (\ref{probKingman11}) that for all $i\in\mathbb{N}^+$,
$$ P(K_{W_1+W_2}(Y)=i)=\frac{(i+1)i}{i(i-1)+2}P(K_{W_1+W_2}(Y)=i+1)+\frac{2}{i(i-1)+2}P(W_1+W_2=i).$$
It is easy to check that for all $i\in\mathbb{N}^+$,
$$\frac{2i-1}{3^{i}}=\frac{(i+1)i}{i(i-1)+2}\frac{2(i+1)-1}{3^{i+1}}+\frac{2}{i(i-1)+2}P(W_1+W_2=i).$$
If $P(K_{W_1+W_2}(Y)=1)<(2-1)/3^1$, then by induction, we have $P(K_{W_1+W_2}(Y)=i)<(2i-1)/3^{i}$ for all $i\in\mathbb{N}^+$. It follows that 
$$\sum_{i=1}^\infty P(K_{W_1+W_2}(Y)=i)<\sum_{i=1}^\infty\frac{2i-1}{3^{i}}=1,$$
which contradicts $\sum_{i=1}^\infty P(K_{W_1+W_2}(Y)=i)=1$. We can generate a similar contradiction if $P(K_{W_1+W_2}(Y)=1)>(2-1)/3^1$. Therefore, we have $P(K_{W_1+W_2}(Y)=1)=(2-1)/3^1$, and by induction $P(K_{W_1+W_2}(Y)=i)=(2i-1)/3^{i}$ for all $i\in\mathbb{N}^+$, which implies 
$$\mu_1^*(\{i\})=\frac{2i-1}{3^{i}}, \forall i\in\mathbb{N}^+$$
is a solution to the recursive distributional equation (\ref{First}). By Theorem \ref{existunique}, the solution to the RDE is unique. 
The result then follows.
\end{proof}

We have the nice form of the $\mu_c^*$ in Theorem 1 only when $c=1$. We do not have a closed-form expression for $\mu_c^*$ in any other cases.

\section*{Acknowledgements}
The author would like to express sincere gratitude and appreciation to Professor Jason Schweinsberg for his invaluable guidance and unwavering encouragement. His methodological support, constructive suggestions and valuable feedback have played an important role in the development of this paper.

\bibliographystyle{abbrv}
\bibliography{references.bib}

\end{document}